\documentclass{article}

\usepackage{graphicx}

\usepackage{latexsym}
\usepackage{amsfonts}
\usepackage[psamsfonts]{amssymb}

\usepackage{calrsfs}

\usepackage{enumerate}
\usepackage{mathtools}

\RequirePackage[OT1]{fontenc}
\RequirePackage[noinfoline,amsthm%,amsmath
]{imsart}

\usepackage{url}

\usepackage{tocvsec2}

\usepackage[unicode%,psdextra
]{hyperref}

\theoremstyle{plain} 
\newtheorem{theorem}{Theorem}[section]
\newtheorem{corollary}[theorem]{Corollary}
\newtheorem{lemma}[theorem]{Lemma}
\newtheorem{proposition}[theorem]{Proposition}

\theoremstyle{definition} 
\newtheorem{definition}[theorem]{Definition}
\theoremstyle{definition} 

\theoremstyle{remark} 

\theoremstyle{remark} 
\newtheorem{remark}[theorem]{Remark}
\newtheorem*{remark*}{Remark}
\numberwithin{equation}{section}

\newcommand{\Det}{\operatorname{det}}
\newcommand{\res}{\operatorname{res}}
\newcommand{\orl}{\operatorname{or}}
\newcommand{\bd}{\operatorname{bd}}
\newcommand{\inter}{\operatorname{int}}

\newcommand{\pd}[1]{\partial_{#1}}

\newcommand{\intr}[2]{\overline{#1,#2}}

\newcommand{\p}{\bar P}

\renewcommand{\tt}[1]{\texttt{#1}}

\newcommand{\supp}{\operatorname{supp}}

\newcommand{\dd}{\mathrm{d}}

\renewcommand{\le}{\leqslant}
\renewcommand{\ge}{\geqslant}
\renewcommand{\leq}{\leqslant}

\newcommand{\m}{\mathbf{m}}
\newcommand{\s}{\mathbf{s}}

\renewcommand{\S}{}

\newcommand{\al}{\alpha}
\newcommand{\be}{\beta}
\newcommand{\ka}{\kappa}

\newcommand{\ga}{\gamma}

\newcommand{\la}{\lambda}
\newcommand{\de}{\delta}
\newcommand{\De}{\Delta}
\newcommand{\vpi}{\varphi}
\renewcommand{\Psi}{\overline{\Phi}}
\newcommand{\Si}{\Sigma}
\newcommand{\Th}{\Theta}
\newcommand{\om}{\omega}
\newcommand{\Om}{\Omega}

\newcommand{\tD}{\tilde D}
\newcommand{\tP}{\tilde P}

\newcommand{\EE}{\mathcal{E}}
\newcommand{\F}{\mathcal{F}}

\newcommand{\SSS}{\mathcal{S}}

\newcommand{\ii}[1]{\boldsymbol{I}\!\left\{#1\right\}}
\renewcommand{\ii}[1]{\operatorname{\mathrm{I}}\!\left\{#1\right\}}

\renewcommand{\P}{\operatorname{\mathsf{P}}} 

\newcommand{\E}{\operatorname{\mathsf{E}}}

\newcommand{\Z}{\mathbb{Z}}
\newcommand{\R}{\mathbb{R}}

\newcommand{\A}{\mathcal{A}}

\newcommand{\vp}{\varepsilon}

\begin{document}

%\protect\footnote{\protect\filename}

\begin{frontmatter}

\title{Optimal binomial, Poisson, and normal left-tail domination for sums of nonnegative random variables}
%\date{\today}

\begin{aug}
\author{\fnms{Iosif} \snm{Pinelis}\ead[label=e1]{ipinelis@math.mtu.edu}}
\runauthor{Iosif Pinelis}

%\affiliation{Michigan Technological University}

\address{Department of Mathematical Sciences\\
Michigan Technological University\\
Houghton, Michigan 49931, USA\\
E-mail: \printead[ipinelis@mtu.edu]{e1}}
\end{aug}

%\title[Inequalities for sums of asymmetric random variables]{\Large Exact inequalities for sums of asymmetric random variables,\\ with applications}
%\date{\today}
%\author{Iosif Pinelis}
%\address{ Department of Mathematical Sciences\\
%Michigan Technological University\\
%Houghton, Michigan 49931 }
%\email{ipinelis@math.mtu.edu}
%\keywords{}
%\subjclass[2000]{Primary: 60E15, 60G50, 60G42, 60G48, 62F03, 62F25, 62G10, 62G15; Secondary: 60E05, 62E10, 62G35}

%\maketitle

%\tableofcontents

%\end{frontmatter}

%\begin{spacing}{1.5}

\begin{abstract}
Let $X_1,\dots,X_n$ be independent nonnegative random variables (r.v.'s), with $S_n:=X_1+\dots+X_n$ and finite values of $s_i:=\operatorname{\mathsf{E}} X_i^2$ and $m_i:=\operatorname{\mathsf{E}} X_i>0$.  
Exact upper bounds on $\operatorname{\mathsf{E}} f(S_n)$  
for all functions $f$ in a certain class $\F$ of \emph{nonincreasing} functions are obtained, in each of the following settings: (i) $n,
m_1,\dots,m_n,s_1,\dots,s_n$ are fixed; (ii)~$n$, $m:=
m_1+\dots+m_n$, and $s:=
s_1+\dots+s_n$ are fixed; (iii)~only $m$ and $s$ are fixed. 
These upper bounds are of the form $\operatorname{\mathsf{E}} f(\eta)$ for a certain r.v.\ $\eta$. 
The r.v.\ $\eta$ and the class $\F$ depend on the choice of one of the three settings. In particular, $(m/s)\eta$ has the binomial distribution with parameters $n$ and $p:=m^2/(ns)$ in setting (ii) and the Poisson distribution with parameter $\la:=m^2/s$ in setting (iii). One can also let $\eta$ have the normal distribution with mean $m$ and variance $s$ in any of these three settings. In each of the settings, the class $\F$ contains, and is much wider than, the class of all decreasing exponential functions. 
As corollaries of these results, optimal in a certain sense upper bounds on the left-tail probabilities $\operatorname{\mathsf{P}}(S_n\le x)$ are presented, for any real $x$. 
In fact, more general settings than the ones described above are considered. 
Exact upper bounds on the exponential moments $\operatorname{\mathsf{E}}\exp\{hS_n\}$ for $h<0$, as well as the corresponding exponential bounds on the left-tail probabilities, were previously obtained by Pinelis and Utev. It is shown that the new bounds on the tails are substantially better. 
\end{abstract}

\begin{keyword}[class=AMS]
\kwd[Primary ]{60E15}
\kwd[; secondary ]{60G42}
\kwd{60G48}
\end{keyword}

% 	60G42   	Martingales with discrete parameter
%		60G44   	Martingales with continuous parameter
%		60G46   	Martingales and classical analysis
%		60G48   	Generalizations of martingales

\begin{keyword}
\kwd{probability inequalities}
\kwd{sums of random variables}
\kwd{submartingales}
\kwd{martingales}
\kwd{upper bounds}
\kwd{generalized moments}
\end{keyword}

%\pagestyle{myheadings} %\markboth{}

%\thispagestyle{filename}

%\markboth{}{}

%\thispagestyle{myheadings} \markboth{}{}

\end{frontmatter}

\settocdepth{chapter}

\tableofcontents 
%{\small\tableofcontents} 

\settocdepth{subsubsection}

%\section{Introduction}\label{intro}

\section{Introduction}\label{intro}
%stock price
Let $X_1,\dots,X_n$ be independent real-valued random variables (r.v.'s), 
with 
\begin{equation*}
	S_n:=X_1+\dots+X_n. 
\end{equation*} 
Exponential upper bounds for $S_n$ go back at least to Bernstein. 
As the starting point here, one uses the multiplicative property of the exponential function together with the condition of independence of $X_1,\dots,X_n$ to write 
\begin{equation}\label{eq:factor}
	\E e^{hS_n}=\prod_1^n\E e^{hX_i}
\end{equation}
for all real $h$. 
Then one  
bounds up each factor $\E e^{hX_i}$, thus obtaining an upper bound (say $M_n(h)$) on $\E e^{hS_n}$, uses the Markov inequality to write $\P(S_n\ge x)\le e^{-hx}\E e^{hS_n}\le B_n(h,x):=e^{-hx}M_n(h)$ for all real $x$ and all nonnegative real $h$, and finally tries to minimize $B_n(h,x)$ in $h\ge0$ to obtain an upper bound on the tail probability $\P(S_n\ge x)$.  
  
This approach was used and further developed in a large number of papers, including notably the well-known work by Bennett \cite{bennett} and Hoeffding \cite{hoeff63}. 
Pinelis and Utev \cite{pin-utev-exp} offered a general approach to obtaining exact bounds on the exponential moments $\E e^{hS_n}$, with a number of particular applications. 

Exponential bounds were obtained in more general settings as well, where the r.v.'s $X_1,\dots,X_n$ do not have to be independent or real-valued. It was already mentioned by Hoeffding at the end of Section~2 in \cite{hoeff63} that his results remain valid for martingales. Exponential inequalities with optimality properties for vector-valued $X_1,\dots,X_n$ were obtained e.g.\ in \cite{pin-sakh,pin94} and then used in a large number of papers. 

Related to this is work on Rosenthal-type and von Bahr--Esseen-type bounds, that is, bounds on absolute power moments $\E|S_n|^p$ of $S_n$; see e.g.\ \cite{bahr65,rosenthal,pin94,latala-moments,bouch-etal,ibr-shar-pos,osek10,tyurinSPL,p=3_publ,rosenthal_AOP,bahr-esseen-AFA}. 

However, the classes of exponential functions $e^{h\cdot}$ and absolute power functions $|\cdot|^p$ are too narrow in that the resulting bounds on the tails are not as good as one could get in certain settings. It is therefore natural to try to consider wider classes of moment functions and then try to choose the best moment function in such a wider class to obtain a better bound on the tail probability. This approach was used and developed in \cite{eaton1,eaton2,T2,pin98,bent-ap,pin-hoeff-published}, in particular. 
The main difficulty one needs to overcome working with such, not necessarily exponential, moment functions is the lack of multiplicative property \eqref{eq:factor}.  

In some settings, the bounds can be improved if it is known that the r.v.'s $X_1,\dots,X_n$ are nonnegative; see e.g.\ \cite{latala-moments,bouch-etal,ibr-shar-pos,osek10}. However, in such settings the focus has usually been on bounds for the right tail of the distribution of $S_n$. 
There has been comparatively little work done concerning the left tail of the distribution of the sum $S_n$ of nonnegative r.v.'s $X_1,\dots,X_n$. 

%
%\cite{hoeff63,pin-sakh,pin94,pin98,bent-ap,pin-hoeff} 
%
%upper bounds on the right tail, among more recent work: \cite{feige,he-zhang-zhang}
%
%power moments (Rosenthal-type bounds): \\  \cite{rosenthal,pin94,latala-moments,bouch-etal,ibr-shar-pos,osek10,p=3_publ,Tyurin,rosenthal_AOP}, especially \cite{latala-moments,bouch-etal,ibr-shar-pos,osek10} for nonnegative r.v.'s
%
%few results concerning the left tail of the sum on nonnegative r.v.'s; quote PU result?

One such result was obtained in \cite{pin-utev-exp}. 
Suppose indeed that the independent r.v.'s $X_1,\dots,X_n$ are nonnegative. 
Also, suppose here that 
\begin{equation}\label{eq:m and s}
\text{$m:=\E X_1+\dots+\E X_n>0$\quad and\quad $s:=\E X_1^2+\dots+\E X_n^2<\infty$.}	
\end{equation}
Then \cite[Theorem~7]{pin-utev-exp} for any $x\in(0,m]$ 
\begin{equation}\label{eq:pin-utev1}
	\P(S_n\le x)\le\exp\Big\{-\frac{m^2}s\,\Big(1+\frac xm\,\ln\frac x{em}\Big)\Big\}
	\le\exp\Big\{-\frac{(x-m)^2}{2s}\Big\}   
\end{equation}
(in fact, these inequalities were stated in \cite{pin-utev-exp} in the equivalent form for the non-positive r.v.'s $-X_1,\dots,-X_n$). 
These upper bounds on the tail probability $\P(S_n\le x)$ were 
based on exact upper bounds on the exponential moments of the sum $S_n$, which can be written as follows: 
\begin{equation}\label{eq:pin-utev2}
	\E\exp\{hS_n\}\le\E\exp\big\{h\tfrac sm\Pi_{m^2/s}\big\}\le\E\exp\big\{h\big(m+Z\sqrt s\,\big)\big\}
\end{equation}
for all real $h\le0$. 
Here and subsequently, 
for any $\la\in(0,\infty)$, let $\Pi_\la$ and $Z$ stand for any r.v.\ having the Poisson distribution with parameter $\la\in(0,\infty)$ and for any standard normal r.v., respectively. 
The bounds in \eqref{eq:pin-utev1} and \eqref{eq:pin-utev2} have certain optimality properties, and they are very simple in form. Yet, they have apparently been little known; in particular, the last bound in \eqref{eq:pin-utev1} was rediscovered in \cite{maurer}. 

%Despite being optimal in a certain sense, 
In the present paper, the ``Poisson'' and ``normal'' bounds in \eqref{eq:pin-utev2} will be extended to a class of moment functions much wider than the ``exponential'' class (still with the preservation of the optimality property, for each moment function in the wider class). Consequently, the bounds in \eqref{eq:pin-utev1} will be much improved. 
We shall also provide ``binomial'' upper bounds on the moments and tail probabilities of $S_n$, which are further improvements of the corresponding ``Poisson'', and hence ``normal'', bounds.  

\section{Summary and discussion}\label{results}

Let $X_1,\dots,X_n$ be nonnegative real-valued r.v.'s. In general, we shall no longer assume that $X_1,\dots,X_n$ are independent; instead, a more general condition, described in the definition below, will be assumed. Moreover, the condition \eqref{eq:m and s} will be replaced by a more general one. 

\begin{definition}\label{def:SS}
Given any $\m=(m_1,\dots,m_n)$ and $\s=(s_1,\dots,s_n)$ in $[0,\infty)^n$, let us say that the r.v.'s $X_1,\dots,X_n$ satisfy the $\S(\m,\s)$-condition if, for some filter $(\A_0,\dots,\A_{n-1})$ of sigma-algebras and each $i\in\intr1n$, the r.v.\ $X_i$ is $\A_i$-measurable, 
\begin{equation}\label{eq:>m_i,<s_i}
	\E(X_i|\A_{i-1})\ge m_i,\quad\text{and}\quad\E(X_i^2|\A_{i-1})\le s_i. 
\end{equation}
Given any nonnegative $m$ and $s$, 
let us also say that the $\S(m,s)$-condition is satisfied if the $\S(\m,\s)$-condition holds for some $\m=(m_1,\dots,m_n)$ and $\s=(s_1,\dots,s_n)$ in $[0,\infty)^n$ such that 
\begin{equation}\label{eq:>m,<s}
	m_1+\dots+m_n\ge m\quad\text{and}\quad s_1+\dots+s_n\le s. 
\end{equation}
\end{definition}

In the above definition and in what follows, for any $\al$ and $\be$ in $\Z\cup\{\infty\}$, we let $\intr\al\be:=\{j\in\Z\colon\al\le j\le\be\}$.  

The following comments are in order. 

\begin{itemize}
	\item
Any independent r.v.'s $X_1,\dots,X_n$ satisfy the $\S(\m,\s)$-condition if $\E X_i\ge m_i$ and $\E X_i^2\le s_i$ for each $i\in\intr1n$; if at that \eqref{eq:>m,<s} holds, 
%$m_1+\dots+m_n\ge m$ and $s_1+\dots+s_n\le s$, 
then the $\S(m,s)$-condition holds as well. 
	\item 
If r.v.'s $X_1,\dots,X_n$ satisfy the $\S(\m,\s)$-condition, then the r.v.'s $X_1-m_1,\dots,\break 
X_n-m_n$ are submartingale-differences, with respect to the corresponding filter $(\A_0,\dots,\A_{n-1})$.  
	\item If, for some $\m$ and $\s$ in $[0,\infty)^n$, the $\S(\m,\s)$-condition is satisfied by some %\emph{nonnegative} 
r.v.'s $X_1,\dots,X_n$, then necessarily 
\begin{equation}\label{eq:s>m^2}
	s_i\ge m_i^2\quad\text{for all }i\in\intr1n. 
\end{equation}
Moreover, if, for some nonnegative $m$ and $s$, the $\S(m,s)$-condition is satisfied by some %\emph{nonnegative} 
r.v.'s $X_1,\dots,X_n$, then necessarily 
\begin{equation}\label{eq:s>m^2/n}
	\tfrac sn\ge\big(\tfrac mn\big)^2 \quad\text{or, equivalently,}\quad n\ge\tfrac{m^2}s.  
\end{equation}
\end{itemize}

\begin{definition}\label{def:Y} 
Given any real numbers $m$ and $s$ such that $m>0$ and $s\ge m^2$ (cf.\ \eqref{eq:s>m^2}),  
%and \eqref{eq:s>m^2/n}), 
let $Y^{m,s}$ stand for any r.v.\ such that 
\begin{equation*}
	\E Y^{m,s}=m,\quad\E(Y^{m,s})^2=s,\quad\text{and}\quad\P(Y^{m,s}\in\{0,\tfrac sm\})=1; 
\end{equation*}
such a r.v.\ $Y^{m,s}$ exists, and its distribution is uniquely determined: 
\begin{equation*}
	\P(Y^{m,s}=\tfrac sm)=1-\P(Y^{m,s}=0)=\tfrac{m^2}s; 
\end{equation*}  
moreover, let $Y_1^{m,s},\dots,Y_n^{m,s}$ denote independent copies of a r.v.\ $Y^{m,s}$.  
Also, given any $\m$ and $\s$ in $(0,\infty)^n$ such that the condition \eqref{eq:s>m^2} holds, we shall always assume the corresponding r.v.'s $Y^{m_1,s_1},\dots,Y^{m_n,s_n}$ to be independent. 
\end{definition}

%For any $\la\in(0,\infty)$, let $\Pi_\la$ stand for any r.v.\ having the Poisson distribution with parameter $\la$. Let $Z$ stand for any standard normal r.v. 

\bigskip

Next, let us describe the pertinent classes of generalized moment functions. 
For any natural $j$, let $\SSS^j$ denote the class of all 
$(j-1)$-times differentiable functions $g\colon\R\to\R$ such that the $(j-1)$th derivative  
$g^{(j-1)}$ of $g$ has a right-continuous right derivative, which will be denoted here simply by $g^{(j)}$. 
As usual, we let $g^{(0)}:=g$.
Take then any natural 
$$k\le j+1$$ 
and introduce the class of functions  
\begin{equation}\label{eq:F}
	\F_+^{k:j}:=
	\big\{g\in\SSS^j\colon\text{$g^{(i)}$ is nondecreasing for each $i\in\intr{k-1}j$\,}\big\}
\end{equation}
and, finally, the ``reflected'' class 
\begin{equation}\label{eq:F_-}
	\F_-^{k:j}:=\{g^-\colon g\in\F_+^{k:j}\}, 
\end{equation}
where $g^-(x):=g(-x)$ for all $x\in\R$. 
%
%For any natural $m$, let $\C^m$ denote the class of all $m$ times continuously differentiable functions 
%$f\colon\R\to\R$.
%Consider then the following classes of functions: 
%\begin{align}
%	\FF2_-&:=\{f\in\C^1\colon \text{$f$, $-f'$, $f''$ are nonincreasing}\}, \label{eq:F2}\\
%	\FF3_-&:=\{f\in\C^2\colon \text{$f$, $-f'$, $f''$, and $-f'''$ are nonincreasing}\}. \label{eq:F3} 
%\end{align}
%Note that the second derivative $f''$ does not have to exist in the usual sense for $f\in\C^1$; the condition in \eqref{eq:F2} that $f''$ is nonincreasing is used as a shortcut for ``$f'$ is concave and $f''$ denotes the left derivative of $f'$\,''. 
%Similarly, the condition in \eqref{eq:F3} that $-f'''$ is nonincreasing is used as a shortcut for ``$f''$ is convex and $f'''$ denotes the left derivative of $f''$\,''. 
It is clear that the class $\F_-^{k:j}$ gets narrower as $j$ increases (with a fixed $k$), and it gets wider as $k$ increases (with a fixed $j$). 

As an example, the 
function 
$x\mapsto a+b\,x+c\,e^{-\la x}$ belongs to $\F_-^{k:j}$ 
for any $a\in\R$, $b\le0$, $c\ge0$, $\la\ge0$ (and any natural $k$ and $j$ such that $k\le j+1$).  
Also, 
given any $a\in\R$, $b\le0$, $c\ge0$, and $w\in\R$, the 
function $x\mapsto a+b\,x+c\,(w-x)_+^\al$ 
belongs to $\F_-^{k:j}$ for any real $\al\ge k$ (and any natural $k$ and $j$ such that $k\le j+1$);  
here and elsewhere, as usual,  
$x_+:=\max(0,x)$ and $x_+^\al:=(x_+)^\al$ for $x\in\R$. 
Note also that the classes $\F_-^{k:j}$ are convex cones; that is, any linear combination with nonnegative coefficients of functions belonging to any one of these classes belongs to the same class. 
%Moreover, 
%we shall be using  
%
\begin{remark}\label{rem:shift}
It 
is not difficult to see that, 
if a function $f$ is in the class $\F_-^{k:j}$, 
then the shifted and/or rescaled function 
$x\mapsto f(bx+a)$ is also in the same class, for any constants $a\in\R$ and $b\ge0$. 
That is, these classes of functions are shift- and scale-invariant. 
\end{remark}

%\begin{proposition}\label{prop:f at -infty}
%If $f\in\FF2_-$, then $f(x)=O(x)$ as $x\to\infty$. 
%\end{proposition} 

%The necessary proofs will be given in Section~\ref{proofs}. 

%It follows from Proposition~\ref{prop:f at -infty} that $\E f(X)$ exists and is finite for any nonnegative r.v.\ $X$ with a finite mean and any $f\in\FF2_-$ (and hence any $f\in\FF3_-$). 

%Recall the definition of the
%{\em Schur majorizarion\/}: .......

%!!!!!!!!!!!!!!! existence of $\E f(S_n)$ etc. 

%\newpage
\bigskip

Now we are ready to state the main result of this paper. 

\begin{theorem}\label{th:}\ 
% Take any $\m$ and $\s$ in $(0,\infty)^n$ such that the condition \eqref{eq:s>m^2} holds. 
\begin{enumerate}[(I)]
	\item 
	Let $X_1,\dots,X_n$ be any \emph{nonnegative} r.v.'s satisfying the $\S(\m,\s)$-condition for some $\m$ and $\s$ in $(0,\infty)^n$, so that \eqref{eq:s>m^2} holds. Then 
\begin{equation}\label{eq:th(I)}
	\E f(S_n)\le\E f\big(Y^{m_1,s_1}+\dots+Y^{m_n,s_n}\big) 
\end{equation}
for all $f\in\F_-^{1:2}$. 
%If, in addition, $\E(X_i|\A_{i-1})=m_i$ for all $i$, then \eqref{eq:th(I)} holds for all $f\in\F_-^{2:2}$; if, furthermore, $\E(X_i^2|\A_{i-1})=s_i$  for all $i$ and at that either $f$ is bounded from below or the $X_i$'s are all bounded from above, then \eqref{eq:th(I)} holds for all $f\in\F_-^{3:2}$. 
	\item 
	Let $X_1,\dots,X_n$ be any \emph{nonnegative} r.v.'s satisfying the $\S(m,s)$-condition for some $m$ and $s$ in $(0,\infty)$, so that \eqref{eq:s>m^2/n} holds. 
%	Let $Y_1^{m/n,s/n},\dots,Y_n^{m/n,s/n}$ denote independent copies of a r.v.\ $Y^{m/n,s/n}$. 
	Then 
\begin{align}
	\E f(S_n)&\le\E f\big(Y_1^{\frac mn,\frac sn}+\dots+Y_n^{\frac mn,\frac sn}\big) \label{eq:th(II,1)}\\ 
	&\le\E f\big(\tfrac sm\Pi_{m^2/s}\big) \label{eq:th(II,2)}\\ 
	&\le\E f\big(m+Z\sqrt s\,\big)  \label{eq:Z,3}
\end{align}
for all $f\in\F_-^{1:3}$; 
in fact, \eqref{eq:Z,3} and the %essentially more general 
inequality 
\begin{equation}\label{eq:Z,2}
		\E f(S_n)\le\E f\big(m+Z\sqrt s\,\big) 
\end{equation}
both hold for all $f\in\F_-^{1:2}$.  
\end{enumerate}
\end{theorem} 

The necessary proofs will be given in Section~\ref{proofs}. 

\begin{remark}\label{rem:conds}
Under the corresponding conditions given in Theorem~\ref{th:}, 
the expected values in inequalities \eqref{eq:th(I)}--\eqref{eq:Z,2} exist (in $\R$ or, at least, in $(-\infty,\infty]$),  according to \cite[Proposition 5.2, part~(i)]{cones}. 
Moreover, the conditions for \eqref{eq:th(I)}--\eqref{eq:Z,2} in Theorem~\ref{th:} can be supplemented or relaxed as follows. 
To describe these extended or relaxed conditions for \eqref{eq:th(I)}--\eqref{eq:Z,2}, introduce the conditions of equalities in 
\eqref{eq:>m_i,<s_i} and/or \eqref{eq:>m,<s}: 
\begin{align}
	\E(X_i|\A_{i-1})&=m_i\quad\text{for all }i, \label{eq:=m_i} \\ 
	\E(X_i^2|\A_{i-1})&=s_i\quad\text{for all }i, \label{eq:=s_i} \\ 
	m_1+\dots+m_n&=m, \label{eq:=m} \\ 
		s_1+\dots+s_n&=s \label{eq:=s}
\end{align}
and also conditions 
\begin{gather}
	\text{the $X_i$'s are bounded or $f\ge p$ for some quadratic polynomial $p$}, \label{eq:except} \\ 
	\E X_i^3<\infty\ \text{for all }i. \label{eq:E X_i^3}
\end{gather}

Then 
\begin{enumerate}[(I)]
	\item inequalities \eqref{eq:th(I)} and \eqref{eq:Z,2} hold if any one of the following two conditions holds: 
	\begin{enumerate}[(i)]
	\item \eqref{eq:=m_i} and $f\in\F_-^{2:2}$; 
	\item \eqref{eq:=m_i}, \eqref{eq:=s_i}, \eqref{eq:except}, and $f\in\F_-^{3:2}$.  
\end{enumerate}
	\item inequality \eqref{eq:th(II,1)} holds if any one of the following three conditions holds: 
	\begin{enumerate}[(i)]
	\item \eqref{eq:=m_i} and $f\in\F_-^{2:3}$; 
		\item \eqref{eq:=m_i}, \eqref{eq:=s_i}, \eqref{eq:except}, and $f\in\F_-^{3:3}$;
	\item \eqref{eq:=m_i}, \eqref{eq:=s_i}, \eqref{eq:E X_i^3}, and $f\in\F_-^{4:3}$.  
\end{enumerate}
	\item inequality \eqref{eq:th(II,2)} holds if any one of the following three conditions holds: 
	\begin{enumerate}[(i)]
	\item $f\in\F_-^{2:3}$; 
		\item \eqref{eq:except} and $f\in\F_-^{3:3}$;
	\item \eqref{eq:E X_i^3} and $f\in\F_-^{4:3}$.  
\end{enumerate}
	\item inequality \eqref{eq:Z,3} holds if any one of the following two conditions holds: 
	\begin{enumerate}[(i)]
	\item $f\in\F_-^{2:2}$; 
	\item \eqref{eq:except} and $f\in\F_-^{3:2}$.  
\end{enumerate}
\end{enumerate}
This remark can be verified similarly to Theorem~\ref{th:}. 
\end{remark}

Obviously, the r.v.'s $Y^{m_1,s_1},\dots,Y^{m_n,s_n}$ in \eqref{eq:th(I)} satisfy the $\S(\m,\s)$-condition. So, inequality \eqref{eq:th(I)} is exact, in the sense that, given any natural $n$ and any $\m$ and $\s$ in $(0,\infty)^n$ such that \eqref{eq:s>m^2} holds, the right-hand side of \eqref{eq:th(I)} is the exact upper bound on its left-hand side. 
Similarly, given any natural $n$ and any $m$ and $s$ in $(0,\infty)$ such that \eqref{eq:s>m^2/n} holds, inequality \eqref{eq:th(II,1)} is exact. 

\begin{proposition}\label{prop:exact}\ 
%\begin{enumerate}[(i)]
%	\item Given any $m$ and $s$ in $(0,\infty)$ and any natural $n$, inequality \eqref{eq:th(II,1)} is exact. 
%	\item 
Given any $m$ and $s$ in $(0,\infty)$, the Poisson upper bound in \eqref{eq:th(II,2)} on $\E f(S_n)$ is exact \big(in this case $n$ is not fixed, having only to satisfy \eqref{eq:s>m^2/n}\big). 
%\end{enumerate}
\end{proposition}

Inequality \eqref{eq:Z,2} is best possible in the following limited sense, at least. 
By \cite[Corollary~5.9]{cones}, this inequality holds for all $f\in\F_-^{1:2}$ if and only if it holds for all 
%Proposition~\ref{prop:closure}, 
%the class $\FF2_-$ is a closure of the convex cone generated by the constants and the 
functions $f$ of the form $f_{w,2}$ for $w\in\R$, where 
\begin{equation}\label{eq:f_w}
	f_{w,\al}(x):=(w-x)_+^\al. 
\end{equation}
Let now positive $m$ and $s$ vary so that $m^2/s\to\infty$, which is the case e.g.\ when $0\ne m_1=m_2=\cdots$, $0<s_1=s_2=\cdots$, conditions \eqref{eq:=m} and \eqref{eq:=s} hold, and $n\to\infty$. 
At that, fix any real $\ka$ and let $w=m+\ka\sqrt s$. 
Let $L_{m,s;w}:=\E f_{w,2}\big(\tfrac sm\Pi_{m^2/s}\big)$, which is, according to Proposition~\ref{prop:exact}, the exact upper bound on $\E f_{w,2}(S_n)$ given $m$ and $s$. 
Then $L_{m,s;w}\sim\E f_{w,2}\big(m+Z\sqrt s\,\big)$; as usual, $a\sim b$ means that $a/b\to1$. 
Indeed, introducing $\tilde Z:=(\Pi_{m^2/s}-m^2/s)/\sqrt{m^2/s}$, one has $\tilde Z\to Z$ in distribution, so that $\frac1s\,L_{m,s;w}=\E f_{\ka,2}(\tilde Z)\to\E f_{\ka,2}(Z)=\frac1s\,\E f_{w,2}\big(m+Z\sqrt s\,\big)$. 
This convergence is justified, since $f_{\ka,2}(\tilde Z)$ is uniformly integrable (as e.g.\ in \cite[Theorem~5.4]{billingsley}), which in turn follows because for any $\la$ and $\al$ in $(0,\infty)$ one has
$\E\exp\frac{\Pi_\la-\la}{\sqrt\la}=\exp\big\{\la\big(e^{-1/\sqrt\la}-1+1/\sqrt\la\,\big)\big\}\le\sqrt e<\infty$ and 
$f_{\ka,\al}(x)/e^{-x}\to0$ as $x\to-\infty$. 

%Note also that \eqref{eq:th(II,1)} fails to hold in general if the condition $f\in\F_-^{1:3}$ is replaced by (say) $f\in\F_-^{1:2}$; consider e.g.\ $f=f_{1,2}
%
%!!!!!!!!!!!!!!!!! \\ 
%\eqref{eq:th(II,1)} fails for $\H2$;  \verb9p=2_is_bad.nb9 \\ 
%\eqref{eq:Z,2} fails for $\H1$;  \verb9p=1_is_bad_for_normal.nb9 

Let $\eta$ denote an arbitrary real-valued r.v. 
Recalling that for any natural $\al$ and any 
$w\in\R$ the 
function $f_{w,\al}$ 
belongs to $\F_-^{1:\al}$ and applying the Markov inequality, one sees that Theorem~\ref{th:} immediately implies 

\begin{corollary}\label{cor:tails}\ 
	Let $X_1,\dots,X_n$ be any \emph{nonnegative} r.v.'s satisfying the $\S(m,s)$-condition for some $m$ and $s$ in $(0,\infty)$, so that \eqref{eq:s>m^2/n} holds. 
Then 
\begin{align}
	\P(S_n\le x)&\le P_3\big(\Si_{n;m,s};x\big) \label{eq:cor(1)} \\ 
	&\le P_3\big(\Si_{\infty;m,s};x\big)	\label{eq:cor(2)}\\ 
	&\le P_3\big(m+Z\sqrt s;x\big); \label{eq:cor:Z,2}
\end{align}
here and in what follows, $x$ is an arbitrary real number (unless otherwise indicated), 
\begin{align}
	\Si_{n;m,s}&:=Y_1^{\frac mn,\frac sn}+\dots+Y_n^{\frac mn,\frac sn}\quad\text{for natural $n$}, \label{eq:Si_n}\\  
	\Si_{\infty;m,s}&:=\tfrac sm\Pi_{m^2/s}, \label{eq:Si_infty}
\end{align}
and  
%for all $x\in\R$, where 
\begin{equation*}
	P_\al(\eta;x):=\inf_{w\in(x,\infty)}\frac{\E(w-\eta)_+^\al}{(w-x)^\al}
\end{equation*}
for any real $\al>0$. 
Also, the upper bound $P_3\big(m+Z\sqrt s;x\big)$ on $\P(S_n\le x)$ can be somewhat improved: 
%, as follows: with 
%\begin{equation}
%	z:=\frac{x-m}{\sqrt s}, 
%\end{equation}
\begin{equation}\label{eq:cor:Z}
		\P(S_n\le x)\le P_2\big(m+Z\sqrt s;x\big).  %=P_2\big(Z;z\big) 
\end{equation}
%for all $x\in\R$.  
\end{corollary} 

The computation of $P_\al(\eta;x)$ is described (in a somewhat more general setting) in  \cite[Theorem~2.5]{pin98};  
for normal $\eta$, similar considerations were given already in \cite[page~363]{pin94} \big(those descriptions are given for the right tail of $\eta$, so that one will have to make the reflection $x\mapsto-x$ to apply those results\big).   
An elaboration of \cite[Theorem~2.5]{pin98} is presented in \cite[Proposition~3.2]{pin-hoeff-arxiv-reftoAIHP}. 
Concerning fast and effective calculations of the positive-part moments $\E X_+^\al$, see \cite{positive}. 
In \cite{bent-64pp}, one can find specific details on the calculation of $P_\al(\eta;x)$ for $\al\in\{1,2,3\}$ and $\eta$ with a distribution belonging to a common particular family such as binomial and Poisson. 

Let us present here some of those results, which will be useful in this context. 
Take any real $\al>1$ and any r.v.\ $\eta$ such that $\E\eta_-^\al<\infty$; then there exists $\E\eta\in(-\infty,\infty]$. 
%Then 
%\begin{equation}
%	P_\al(a+b\eta;x)=P_\al\big(\eta;\tfrac{x-a}b\big)
%\end{equation}
%for all real $x$ and $a$ and all $b\in(0,\infty)$. 
Let 
\begin{equation}\label{eq:x_*}
	x_*:=x_*(\eta):=\inf\supp(\eta) 
%	\quad\text{and}\quad 
%	x_{**}:=x_{**}(\eta):=\inf\big(\supp(\eta)\setminus\{x_*\}\big)
	, 
\end{equation}
where $\supp(\eta)$ denotes the support set of (the distribution of) the r.v. $\eta$, and 
\begin{equation*}
	\ga(w):=\ga(\eta;w):=\frac{\E\eta\,(w-\eta)_+^{\al-1}}{\E(w-\eta)_+^{\al-1}}
\end{equation*}
for $w\in(x_*,\infty)$. 
Then, by \cite[Proposition~3.2]{pin-hoeff-arxiv-reftoAIHP}, 
the function $\ga$ is continuous and nondecreasing on the interval $(x_*,\infty)$ and for every 
$x\in(x_*,\E\eta)$ there exists a unique $w_x = w_{x;\al,\eta}\in(x_*,\infty)$ such
that
\begin{equation*}
	\ga(w_x) = x;  
\end{equation*}
in fact, $w_x\in(x,\infty)$. 
It follows that, for every 
$x\in(x_*,\E\eta)$,  
\begin{equation}\label{eq:EE}
	\EE_{\al;x}(w):=\E(w-\eta)_+^{\al-1}(\eta-x)
	\left\{
		\begin{alignedat}{2}
	&<0\quad &&\text{for }w\in(x_*,w_x), \\  
	&=0\quad &&\text{for }w=w_x, \\  
	&>0\quad &&\text{for }w\in(w_x,\infty); 
	\end{alignedat}
	\right.
\end{equation}
in particular, $w_x$ is the only root in $(x_*,\infty)$ of the equation 
\begin{equation}\label{eq:EE=0}
	\EE_{\al;x}(w_x)=0. 
\end{equation}
Also by \cite[Proposition~3.2]{pin-hoeff-arxiv-reftoAIHP}, 
\begin{equation*}
	P_\al(\eta;x)=\left\{
	\begin{alignedat}{2}
	&\P(\eta\le x)=\P(\eta=x)\quad &&\text{for }x\in(-\infty,x_*], \\  
	&\frac{\E^\al(w_x-\eta)_+^{\al-1}}{\E^{\al-1}(w_x-\eta)_+^\al}&&\text{for }x\in(x_*,\E\eta), \\  
		&1&&\text{for }x\in[\E\eta,\infty).   
	\end{alignedat}
	\right.
\end{equation*}
In particular, the upper bound $P_\al(\eta;x)$ on the left-tail probability $\P(\eta\le x)$ is exact for $x\in(-\infty,x_*]$. 

Thus, to evaluate $P_\al(\eta;x)$ for any real $x$, it is enough to find $w_x$ (that is, to solve  equation \eqref{eq:EE=0}) for any 
$x\in(x_*,\E\eta)$. 

This is especially easy to do if the r.v.\ $\eta$ takes values in a lattice, which is the case when $\eta$ is $\Si_{n;m,s}$ or $\Si_{\infty;m,s}$, as in Corollary~\ref{cor:tails}. 
Again by \cite[Proposition~3.2]{pin-hoeff-arxiv-reftoAIHP}, 
\begin{equation*}
	P_\al(a+b\eta;x)=P_\al\big(\eta;\tfrac{x-a}b\big)
\end{equation*}
for all real $x$ and $a$ and all $b\in(0,\infty)$. 
So, the calculation of $P_\al(\eta;x)$ for $\eta$ equal $\Si_{n;m,s}$ or $\Si_{\infty;m,s}$ reduces to the situation when the r.v.\ $\eta$ is integer-valued with $x_*=x_*(\eta)=0$; assume for now that this is the case. 
In view of \eqref{eq:cor(1)} and \eqref{eq:cor(2)}, assume also that $\al=3$. Then, by \eqref{eq:EE}, 
\begin{equation}\label{eq:EE=quad}
	\EE_{3;x}(w):=a_j w^2-2b_j w+c_j, 
\end{equation}	
where $x\in(x_*,\E\eta)=(0,\E\eta)$, $w\in(x_*,\infty)=(0,\infty)$, 
\begin{equation*}
\begin{aligned}	
	j&:=\lceil w-1\rceil\ (\text{so that}\ j\in\intr0\infty\ \text{and}\ j<w\le j+1), \\
	a_j&:=a_{j,x}:=\E(\eta-x)\ii{\eta\le j},\\ 
	b_j&:=b_{j,x}:=\E\eta(\eta-x)\ii{\eta\le j},\\ 
	c_j&:=c_{j,x}:=\E\eta^2(\eta-x)\ii{\eta\le j}. 
\end{aligned}	
\end{equation*}
Therefore and in view of \eqref{eq:EE=0} and \eqref{eq:EE}, for each $x\in(x_*,\E\eta)=(0,\E\eta)$ one finds $w_x$ as the only root in the interval $(j_x,j_x+1]$ of the quadratic equation 
\begin{equation}\label{eq:quadr}
	a_{j_x} w_x^2-2b_{j_x} w_x+c_{j_x}=0,  
\end{equation}
where $j_x:=\min\big\{j\in\intr0\infty\colon a_j\,(j+1)^2-2b_j\,(j+1)+c_j\ge0\big\}$. 
If $a_{j_x}\ne0$ then, by \eqref{eq:EE} and \eqref{eq:EE=quad}, $w_x$ is the greater of the roots of the above quadratic equation.  
%\begin{equation}
%	w_x=\frac{b_j+\sqrt{b_j^2-a_jc_j}}{a_j}
%\end{equation}

%\begin{equation}
%\begin{gathered}
%	\EE_{3;x}(w):=a_j w^2-2b_j w+c_j, \\
%	\text{where}\hfill \\ 
%	j:=\lceil w-1\rceil\ (\text{so that}\ j<w\le j+1), \\
%	a_j:=a_{j,x}:=\E(\eta-x)\ii{\eta\le j},\\ 
%	b_j:=b_{j,x}:=\E\eta(\eta-x)\ii{\eta\le j},\\ 
%	c_j:=c_{j,x}:=\E\eta^2(\eta-x)\ii{\eta\le j}. 
%\end{gathered}	
%\end{equation}

The interesting paper \cite{dance} presents, for any given $n\in\intr0\infty\cup\{\infty\}$ and $\la\in(1,\infty)$, the exact upper bound (say $B_{n,\la}$) on $\P(S\le1)$ under the condition that $S=\sum_{i=1}^n X_i$, where the $X_i$'s are independent r.v.'s such that $0\le X_i\le1$ for all $i\in\intr1n$ 
and $\E S=\la$. 
\big(For $\la\in[0,1]$, the exact upper bound $B_{n,\la}$ is trivial and equals $1$; indeed, let $X_1$ take values $0$ and $1$ with probabilities $1-\la$ and $\la$, respectively, and let $X_i=0$ for all $i\in\intr2n$.\big)
%, where $n\in\N\cup\{\infty\}$ and $\la\in(1,\infty)$ are given. 
%This was done for an arbitrary fixed natural $n$, as well as for the the case when $n$ is allowed to vary freely; the latter case . 
%In particular, $B_{\infty,\la}=e^{-\la}\max(1+\la,e)$. 
Note that the conditions $0\le X_i\le1$ for all $i$ and $\E S=\la$ imply $\sum_i\E X_i=\la$ and $\sum_i\E X_i^2\le\la$, which corresponds to the $\S(m,s)$-condition with $m=s=\la$. 
So, it makes sense to compare the bound $P_3\big(\Si_{n;\la,\la};1\big)$ in \eqref{eq:cor(1)}--\eqref{eq:cor(2)} with $B_{n,\la}%=e^{-\la}\max(1+\la,e)
$. 
Graphs of these two bounds and their ratio in the case $n=\infty$ are shown in Figure~\ref{fig:compar-Dance3}. 

\begin{figure*}[h]
	\centering		\includegraphics[width=1.00\textwidth]{%C:/Users/Iosif/Dropbox/mtu/PU-left-tail-improved/pin-utev-better/paper/
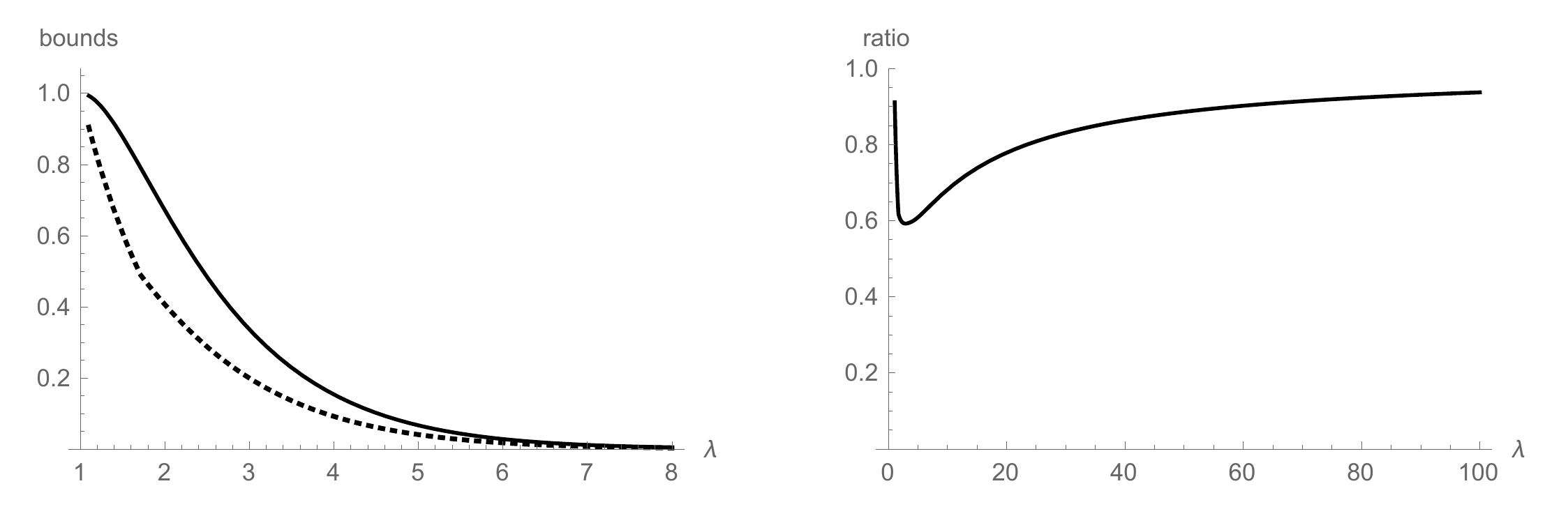}
	\caption{Left panel: graphs $\big\{\big(\la,P_3(\Si_{\infty;\la,\la};1)\big)\colon1.1\le\la\le8\big\}$ (solid) and $\{(\la,B_{\infty,\la})\colon\break 
	1.1\le\la\le8\}$ (dotted). 
	Right panel: graph $\big\{\big(\la,B_{\infty,\la}/P_3(\Si_{\infty;\la,\la};1)\big)\colon1.1\le\la\le100\big\}$. 
	}
	\label{fig:compar-Dance3}
\end{figure*}
\noindent The calculations of $P_3\big(\Si_{\infty;\la,\la};1\big)$ here were done in accordance with the above description, containing formulas \eqref{eq:x_*}--\eqref{eq:quadr}; it takes less than $0.3$ sec with Mathematica on a standard laptop to produce either of the two graphs in Figure~\ref{fig:compar-Dance3}. 
It can be seen that the bound $P_3\big(\Si_{\infty;\la,\la};1\big)$ is not much greater than the optimal bound $B_{\infty,\la}$, especially when $\la$ is close to either $1$ or $\infty$; the corresponding comparisons for finite $n$ look similar. 
On the other hand, our bounds $P_3\big(\Si_{n;m,s};x\big)$ hold under much more general conditions: (i) for all $x\in\R$, rather than just for $x=1$; (ii) assuming only the $(m,s)$-condition (on the sums of the first and second moments of the $X_i$'s), rather than requiring all the $X_i$'s to be bounded by the constant $1$ -- which latter also coincides with the value of $x$ chosen in \cite{dance}; (iii) assuming the more general dependence conditions. 

By \cite[Proposition%s~3.3--
~3.5]{pin-hoeff-arxiv-reftoAIHP}, 
\begin{equation}\label{eq:al to infty}
	P_\al(\eta;x)\uparrow P_\infty(\eta;x):=\inf_{h<0}e^{-hx}\E e^{h\eta}
\end{equation}
%for all $x\in\R$ 
as $\al$ increases from $0$ to $\infty$; 
thus, the bounds $P_\al(\eta;x)$ improve on the so-called exponential bounds $P_\infty(\eta;x)$. 
In particular, letting 
\begin{equation*}
	\la:=\frac{m^2}s\quad\text{and}\quad z:=\frac{x-m}{\sqrt s},   
\end{equation*}
one has  
\big(cf.\ \eqref{eq:cor(1)}, \eqref{eq:cor(2)}, and \eqref{eq:cor:Z}\big), 
\begin{align}
P_2\big(m+Z\sqrt s;x\big)%=P_2(Z;z)
&\le P_\infty\big(m+Z\sqrt s;x\big)=e^{-z^2/2}, \label{eq:<P_infty(Z)} %\quad\text{and} 
\\
%\end{gather} 
%\begin{align}
	P_3\big(\Si_{\infty;m,s};x\big)
	&\le P_\infty\big(\Si_{\infty;m,s};x\big) \label{eq:<P_infty(Pi)} \\ 
	&=\exp\Big\{-\la\Big[\Big(1+\frac z{\sqrt\la}\Big)\ln\Big(1+\frac z{\sqrt\la}\Big)-\frac z{\sqrt\la}\Big]\Big\} \label{eq:P_infty(Pi)}\\
	&\le P_\infty\big(m+Z\sqrt s;x\big), \label{eq:P_infty(Pi)<P_infty(Z)} \\ 
%\end{align}
%\begin{align}
	P_3\big(\Si_{n;m,s};x\big)
	&\le P_\infty\big(\Si_{n;m,s};x\big) \label{eq:<P_infty(Si)} \\ 
	&=\left(\frac{\lambda }{\lambda +z\sqrt{\lambda }}\right)^{\lambda +z\sqrt{\lambda }} \left(\frac{n-\lambda }{n-\lambda -z\sqrt{\lambda }}\right)^{n-\lambda-z\sqrt{\lambda }} \label{eq:P_infty(Si)}\\
	&\le P_\infty\big(\Si_{\infty;m,s};x\big),  \label{eq:P_infty(Si)<P_infty(Pi)}
\end{align}
for natural $n\ge\la$ and $z\in[-\sqrt\la,0)$; for $z=-\sqrt\la$, the expressions in \eqref{eq:P_infty(Pi)}  and \eqref{eq:P_infty(Si)} for $P_\infty\big(\Si_{n;m,s};x\big)$ and $P_\infty\big(\Si_{\infty;m,s};x\big)$ are defined by continuity, as $e^{-\la}$ and $(1-\la/n)^n$, respectively;  
%Note that, in view of \eqref{eq:cor:Z,2}, $P_\infty\big(\tfrac sm\Pi_{m^2/s};x\big)\le P_\infty(Z;z)=e^{-z^2/2}$.  
inequalities \eqref{eq:P_infty(Pi)<P_infty(Z)} and \eqref{eq:P_infty(Si)<P_infty(Pi)} follow by \eqref{eq:al to infty}, \eqref{eq:Si_infty}, \eqref{eq:Z,3}, \eqref{eq:Si_n}, 
%
% \eqref{eq:cor:Z,2} 
and \eqref{eq:th(II,2)}.

The exponential upper bounds \eqref{eq:<P_infty(Z)} and \eqref{eq:<P_infty(Si)} are the same (up to a shift, rescaling, and reflection $x\mapsto-x$) as Hoeffding's bounds in \cite[(2.1) and (2.3)]{hoeff63}, where they were obtained under an additional condition, which can be stated in terms of the present paper as 
\begin{equation}\label{eq:hoeff cond}
	\text{$\P(X_i\le\tfrac sm)=1$ for all $i\in\intr1n$.}  
\end{equation}
Note that \eqref{eq:hoeff cond}, together with the conditions %$\E(X_i|\A_{i-1})=m_i$ for all $i$ and $m_1+\dots+m_n=m$
\eqref{eq:=m_i} and \eqref{eq:=m}, implies the second inequalities in \eqref{eq:>m_i,<s_i} and \eqref{eq:>m,<s} with $s_i:=\frac sm\,m_i$. 

For independent $X_i$'s \big(but without the additional restriction \eqref{eq:hoeff cond}\big), 
the exponential upper bounds %$P_\infty\big(\Si_{\infty;m,s};x\big)$ and $P_\infty\big(m+Z\sqrt s;x\big)$ in \eqref{eq:<P_infty(Pi)} and 
%\eqref{eq:<P_infty(Z)}  
in \eqref{eq:<P_infty(Z)} and \eqref{eq:P_infty(Pi)} 
on $\P(S_n\le x)$ --- as well as the exact upper bound $\E f\big(\tfrac sm\Pi_{m^2/s}\big)$ on 
$\E f(S_n)$ for $f(x)\equiv e^{hx}$ with $h<0$ --- were essentially obtained in \cite[Theorem~7]{pin-utev-exp}. 
Note two mistakes concerning the latter result: (i) in the proof in \cite{pin-utev-exp}, $\psi(u)$ should be replaced by $\psi(hu)$ and (ii) what is presented as the proof of Theorem~7 in \cite{pin-utev-exp} is in fact that of Theorem~8 therein, and vice versa. 
Results of \cite{pin-utev-exp} seem yet relatively unknown, as the bound $e^{-z^2/2}$ on $\P(S_n\le x)$ appeared later in \cite{maurer}. 

By \cite[Theorem~3.11]{pin98} or \cite[Theorem~4]{pin99}, with $c_{\al,0}:=\Gamma(\al+1)(e/\al)^\al$, 
\begin{equation*}
	P_\al(\eta;x)\le c_{\al,0}\,\P(\eta\le x)
\end{equation*}
provided that the tail function $x\mapsto\P(\eta\le x)$ is log-concave. 
Combining this result with the Cantelli inequality, one also has the following upper bound on $\P(S_n\le x)$: 
\begin{equation*}
	W(z):=\min\Big(1,\frac1{1+z^2},\,c_{2,0}\P(Z\le z)\Big); 
\end{equation*}
note that $c_{2,0}=e^2/2=3.69\dots$. 
This bound may serve as an easier to compute and deal with approximation to the better bound $P_2\big(m+Z\sqrt s;x\big)$. 

%\newpage
%
%The bounds $P(z)$, equal $P_\al(\Si_{n;m,s};x)$ 

\begin{figure}[h]
	\centering
\includegraphics[width=1.00\textwidth]
{%C:/Users/Iosif/Documents/mtu_home01-30-10/pin-utev-better/paper/
%grid.eps
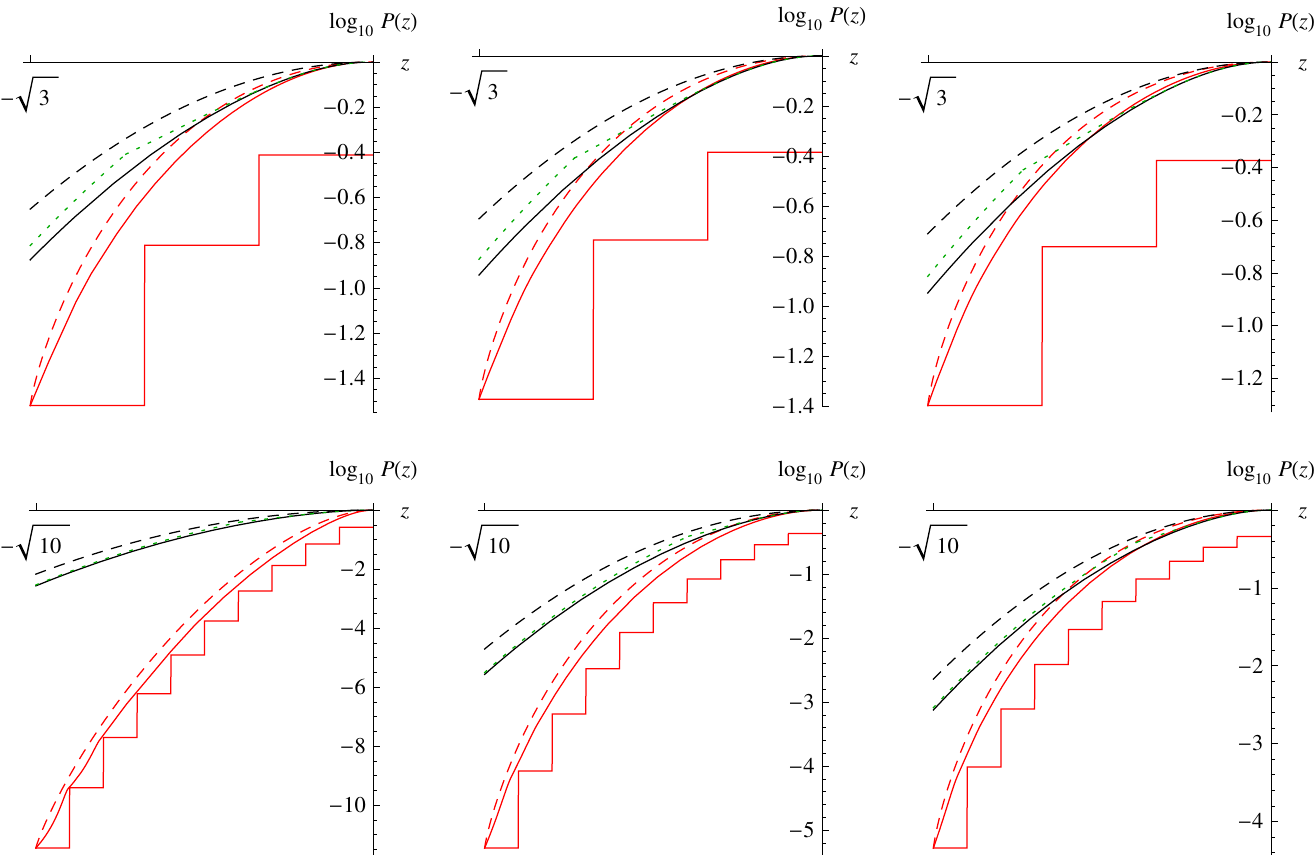}
	\caption{Decimal logarithms of the bounds/tails $P(z)$, for $\la=10$ (first row) and $\la=3$ (second row). The columns correspond to $n=11$ (left), $n=30$ (middle), and $n=\infty$ (right). }
	\label{fig:}
\end{figure}

All the mentioned upper bounds $P(z):=P_\al(\eta;x)$ for $\eta$ equal $\Si_{n;m,s}$ or $m+Z\sqrt s$ %as well as the true left tails $T(z):=\P\big(\Si_{n;m,s}\le x\big)$ 
can be fully expressed in terms of $z$, $\la$, and $n$. These bounds are compared graphically in Figure~\ref{fig:} for $\la\in\{3,10\}$, $\al\in\{0,2,3,\infty\}$, $n\in\{11,30,\infty\}$, and $z\in\big(-\sqrt\la,0\big)$; note that $\P(\Si_{n;m,s}\le x)=P_\al(\Si_{n;m,s};x)=0$ if $z<-\sqrt\la$; here, as is natural, $P_\al(\Si_{n;m,s};x)$ is interpreted as the true tail probability $\P(\Si_{n;m,s}\le x)$ for $\al=0$. 
The graphs of $\log_{10}P_\al(\Si_{n;m,s};x)$ shown in Figure~\ref{fig:} are red: stepwise for $\al=0$, solid-continuous for $\al=3$, and dashed-continuous for $\al=\infty$.  
The graphs of $\log_{10}P_\al\big(m+Z\sqrt s;x\big)$ are black: solid for $\al=2$, and dashed for $\al=\infty$. 
No graphs are shown for $P_\al(\Si_{n;m,s};x)$ with $\al=2$, as those are not established bounds; 
% (cf.\ Remark~\ref{rem:smaller p}); 
nor is there a graph for $P_\al\big(m+Z\sqrt s;x\big)$ with $\al=3$, as the better bound with $\al=2$ is available.   
Also, a graph for $W(z)$ is shown, dotted-green. 

It is seen that the bound $P_3\big(\Si_{n;m,s};x\big)$ is close to the true tail probability $\P\big(\Si_{n;m,s}\le x\big)$, especially for $\la=10$ and $n=11$, with a zero error at the left end-point \big($-\sqrt\la$\,\big) of the range of each of the r.v.\ $\big(\Si_{n;m,s}-m)/\sqrt s$, which is in accordance with part (iv)(b) of the mentioned \cite[Proposition~3.2]{pin-hoeff-arxiv-reftoAIHP}. 
In the latter case ($\la=10$ and $n=11$), the bound $P_3\big(\Si_{n;m,s};x\big)$ is over 8 %orders of magnitude 
times better near the left-end point of the range than the ``normal'' exponential bound $e^{-z^2/2}$. 
However, $P_3\big(\Si_{n;m,s};x\big)$ may be slightly greater for $z$ near $0$ than the ``normal'' better-than-exponential bound $P_2\big(m+Z\sqrt s;x\big)$; this is due to the fact the class $\F_-^{1:2}$ is somewhat richer than $\F_-^{1:3}$.

%\section{Applications}\label{appls} 

\section{Proofs}\label{proofs}

\begin{proof}[Proof of Theorem~\ref{th:}]\ 

\noindent\textbf{(I)}\quad By a standard induction argument (cf.\ e.g.\ \cite[Lemma~12]{asymm}), in order to prove part (I) of the theorem, it is enough to show that \eqref{eq:th(I)} holds for $n=1$. 
%Next, in view of Propositions~\ref{prop:closure} and \ref{prop:f at -infty}, the Lebesgue dominated convergence principle implies that without loss of generality (w.l.o.g.) one may assume $f$ to be in the class $\H2$. 
%So, 
Moreover, by \cite[Corollary~5.9]{cones}, 
we may assume that $f=f_{w,2}$ for some $w\in\R$, where $f_{w,2}$ is defined by formula \eqref{eq:f_w}. 
So, the proof of part (I) will be complete once it is shown that 
\begin{equation}\label{eq:th(I),n=1}
	\E f_{w,2}(X)\le\E f_{w,2}\big(Y^{m,s}\big) 
\end{equation} 
whenever the r.v.\ $X$ is nonnegative, $\E X\ge m$, $\E X^2\le s$, $w\in\R$, and $0<m\le\sqrt s$. 
For $w\le0$, both sides of \eqref{eq:th(I),n=1} are zero. So, w.l.o.g.\ $w>0$. 
Introduce now $z:=\frac sm$, $v:=w\vee z$, and $c:=\frac wv$, and then $g(x):=c^2(v-x)^2$. 
Then $\P(Y^{m,s}\in\{0,z\})=1$, 
$f_{w,2}\le g$ on $[0,\infty)$ and $f_{w,2}=g$ on $\{0,z\}$, whence $g(Y^{m,s})=f_{w,2}(Y^{m,s})$ almost surely (a.s.).  
Note also that $v>0$ and recall the relations $\E X\ge m=\E Y^{m,s}$ and $\E X^2\le s=\E(Y^{m,s})^2$. 
Thus, 
\begin{align*}
	\E f_{w,2}(X)\le\E g(X)&=c^2(v^2-2v\E X+\E X^2) \\
	&\le c^2\big(v^2-2v\E Y^{m,s}+\E(Y^{m,s})^2\big) \\ 
	&=\E g(Y^{m,s})=\E f_{w,2}(Y^{m,s}),  
\end{align*}
which completes the proof of part (I) of Theorem~\ref{th:}. 

\noindent\textbf{(II)}\quad 
Take any $f\in\F_-^{1:3}$ and consider  
\begin{equation*}
	F_{n,f}(P_1,\dots,P_n):=\E f\big(Y^{m_1,s_1}+\dots+Y^{m_n,s_n}\big),
\end{equation*}
the right-hand side of \eqref{eq:th(I)}, where 
\begin{equation}\label{eq:P}
	P_i:=(m_i,s_i)
\end{equation}
for all $i$. 
Note that the function $F_{n,f}$ is symmetric (with respect to all permutations of its $n$ arguments, $P_1,\dots,P_n$). 
Next, if nonnegative r.v.'s $X_1,\dots,X_n$ satisfy the $\S(m,s)$-condition, they satisfy the $\S(\m,\s)$-condition for some $m_1,\dots,m_n,\break 
s_1,\dots,s_n$ such that $m_1+\dots+m_n=m$ and $s_1+\dots+s_n=s$. 
So, by \eqref{eq:th(I)}, to prove \eqref{eq:th(II,1)} it is enough to show that 
\begin{equation}\label{eq:F<F}
	F_{n,f}(P_1,\dots,P_n)\le F_{n,f}(\p_n,\dots,\p_n),
\end{equation}
where $\p_n:=\frac1n(P_1+\dots+P_n)$. 
Here we shall need the following lemma, which establishes a Schur-concavity-like property of the symmetric function $F_{n,f}$. 

\begin{lemma}\label{lem:schur}
For any natural $n\ge2$ and any $t\in[0,1]$
\begin{equation*}
	F_{n,f}(P_1,\dots,P_n)\le F_{n,f}(P_{1+t},P_{2-t},P_3\dots,P_n), 
\end{equation*}
where $P_{1+t}:=(1-t)P_1+tP_2$ and hence $P_{2-t}=tP_1+(1-t)P_2$. 
\end{lemma}

The proof of Lemma~\ref{lem:schur} will be given at the end of this section. 

Note that $F_{n,f}$ is a function of $n$ points $P_1,\dots,P_n$ in $\R^2$, rather than of $n$ real arguments. If the latter were the case, then Lemma~\ref{lem:schur} together with the well-known Muirhead lemma (see e.g.\ \cite[Lemma~2.B.1]{marsh-ol}) would immediately imply the Schur-concavity and hence \eqref{eq:F<F}. 
However, no appropriate ``multidimensional'' analogue of the Muirhead lemma seems to exist. 
Indeed, if one defines the ``multivariate'' majorization by means of doubly stochastic matrices (in accordance with the Hardy-Littlewood-Polya characterization --- see e.g.\ \cite[Theorem~2.B.2]{marsh-ol}), then the analogue of the Muirhead lemma fails to hold. For example, take $n=3$ and consider the doubly stochastic $3\times3$ matrices \big(say $A$ and $B_t$, for some $t\in[0,1]$\big) that transform any triple $\tau:=(Q_1,Q_2,Q_3)$ of points in $\R^2$ to (say)  $\tilde\tau:=\big(\frac{Q_1+Q_2}2,\frac{Q_1+Q_3}2,\frac{Q_2+Q_3}2\big)$ and $\tau_t:=\big((1-t)Q_1+tQ_2,tQ_1+(1-t)Q_2,Q_3\big)$, respectively; 
matrices such as $B_t$ are referred to as $T$-transform matrices, all of which latter can be written as $C^{-1}B_tC$ for some $t\in[0,1]$ and some permutation matrix $C$; see e.g.\ \cite[Section~2.B]{marsh-ol}.  
Then, if the points $Q_1,Q_2,Q_3$ are not collinear, already after one application of any 
matrix $B_t$ with $t\in(0,1)$ 
%nontrivial \big(i.e., with $t\in(0,1)$\big) matrix $B_t$ 
to $\tau$  
one will never be able to get from $\tau_t$ to $\tilde\tau$ via any chain of $T$-transforms, since the points $\frac{Q_1+Q_3}2$ and $\frac{Q_2+Q_3}2$ do not belong to the convex hull of the set $\{(1-t)Q_1+tQ_2,tQ_1+(1-t)Q_2,Q_3\}$. 
%one will not 
%
%For functions of $n$ real arguments, the property corresponding to that of the function $F_{n,f}$ stated in Lemma~\ref{lem:schur} would immediately imply the Schur-concavity 
%; cf.\ the well-known Muirhead lemma; see e.g.\ \cite[Lemma~2.B.1]{marsh-ol}

We shall verify \eqref{eq:F<F} by induction on $n$. For $n=1$, \eqref{eq:F<F} is trivial. 
Suppose that \eqref{eq:F<F} holds for $n$ equal some natural $k$, and consider $n=k+1$. 
Introduce $\tP_k:=\frac1{k+1}\p_k+(1-\frac1{k+1})P_{k+1}$, 
$f_{k+1}(x):=\E f\big(x+Y^{m_{k+1},s_{k+1}}\big)$, and $g_{k+1}(x)
:=\E f\big(x+Y^{\bar m_{k+1},\bar s_{k+1}}\big)$, where $(\bar m_{k+1},\bar s_{k+1}):=\p_{k+1}$. 
By Remark~\ref{rem:shift}, the functions $f_{k+1}$ and $g_{k+1}$ are in $\F_-^{1:3}$. 
Also,  
\begin{gather}\label{eq:tp}
	\tfrac1k\,\big(\tP_k+(k-1)\p_k\big)=\p_{k+1}. 
\end{gather}
It follows that 
\begin{alignat*}{2}
	F_{k+1,f}(P_1,&\dots,P_{k+1})&& \\
	&=\E F_{k,f_{k+1}}(P_1,\dots,P_k) 
	&&\text{\big(by the definition of $f_{k+1}$\big)} \\
	&\le\E F_{k,f_{k+1}}(\p_k,\dots,\p_k) 
		&&\text{\big(by induction\big)} \\
	&=F_{k+1,f}(P_{k+1},\p_k,\dots,\p_k) 
		&&\text{\big(by the definition of $f_{k+1}$} \\
	&%=F_{k+1,f}(P_{k+1},\p_k,\dots,\p_k) 
		&&\text{\ \ and the symmetry of $F_{k+1,f}$\big)} \\
	&\le F_{k+1,f}(\tP_k,\p_{k+1},\p_k,\dots,\p_k) 
			\ \ &&\text{\big(by Lemma~\ref{lem:schur} with $t=\tfrac1{k+1}%\in[0,\tfrac12]
			$\big)} \\
	&=\E F_{k,g_{k+1}}(\tP_k,\p_k,\dots,\p_k) 
		&&\text{\big(by the definition of $g_{k+1}$\big)} \\
	&\le\E F_{k,g_{k+1}}(\p_{k+1},\dots,\p_{k+1}) 
		&&\text{\big(by induction and \eqref{eq:tp}\big)} \\
	&=F_{k+1,f}(\p_{k+1},\dots,\p_{k+1}) 
			&&\text{\big(by the definition of $g_{k+1}$\big)}. 
\end{alignat*}
This completes the proof of \eqref{eq:th(II,1)}, modulo Lemma~\ref{lem:schur}. 

By an argument similar to that used in the proof of part (I) of Theorem~\ref{th:}, it is enough to verify 
\eqref{eq:th(II,2)} and \eqref{eq:Z,3} for $f=f_{w,3}$, and \eqref{eq:Z,2} for $f=f_{w,2}$. 

In inequality \eqref{eq:th(II,1)} with $n+1$ instead of $n$, take $X_{n+1}=0$ and $X_i=Y_i^{\frac mn,\frac sn}$ for $i\in\intr1n$; it then follows that the 
right hand-side of \eqref{eq:th(II,1)} is nondecreasing in $n$, for any fixed positive real $m$ and $s$. 
Next,  
(i) all the r.v.'s in \eqref{eq:th(II,1)} and \eqref{eq:th(II,2)} are nonnegative, (ii) the function $f_{w,3}$ is continuous and bounded on $[0,\infty)$, and (iii) $Y_1^{\frac mn,\frac sn}+\dots+Y_n^{\frac mn,\frac sn}$ converges in distribution to $\tfrac sm\Pi_{m^2/s}$ as $n\to\infty$. 
So, the right hand-side of \eqref{eq:th(II,1)} is, not only nondecreasing in $n$, but also converging to the right hand-side of \eqref{eq:th(II,2)} as $n\to\infty$ (for $f=f_{w,3}$). 
Thus, \eqref{eq:th(II,2)} follows. 

As %stated in part (II) of Theorem~\ref{th:}, 
for inequality \eqref{eq:Z,3}, it is essentially a special case of \eqref{eq:Z,2}.  Indeed, consider the latter inequality with $n\to\infty$ and $X_1=X_1^{(n)},\dots,X_n=X_n^{(n)}$ being independent copies of $c_n\Pi_{\la_n}$, where $c_n:=\frac{s-m^2/n}m\sim\frac sm$ and $\la_n:=\frac{m^2}{ns-m^2}\sim\frac{m^2}{ns}$. 
Then the r.v.'s $X_1,\dots,X_n$ satisfy the $\S(m,s)$-condition, and 
$S_n$ converges to $\tfrac sm\Pi_{m^2/s}$ in distribution. 
Therefore, $\E f_{w,2}(S_n) \longrightarrow \E f_{w,2}\big(\tfrac sm\Pi_{m^2/s}\big)$. 

%\newpage

Thus, it remains to prove \eqref{eq:Z,2}, for $f=f_{w,2}$. 
If at that $w\le0$, then the left-hand side of \eqref{eq:Z,2} is zero, while its right-hand side is nonnegative. 
Therefore and by rescaling, w.l.o.g.\ $w=1$.  
Also, as in the proof of part (I) of Theorem~\ref{th:}, w.l.o.g.\ $n=1$. 
Thus, also in view of \eqref{eq:th(I)} and \eqref{eq:s>m^2/n}, to complete the proof of Theorem~\ref{th:}, it suffices to show that 
\begin{equation*}
	\de(m):=\de(m,k):=\frac{\E\big(1-m-kmZ\,\big)_+^2-\E\big(1-Y^{m,\,k^2m^2}\big)_+^2}
	{2\big(k^2m^2+(m - 1)^2\big)}
	\ge0
\end{equation*}
for all $m\in(0,\infty)$ and $k\in(1,\infty)$. 
Take indeed any $k\in(1,\infty)$. 
Note that % gauss-comparison1.nb 
\begin{equation*}
\de'(m)\,k^2 \left(k^2 m^2+(m-1)^2\right)^2=
\left\{
\begin{aligned}
(D\de)_1(m) & \text{ if }m\in(0,1/k^2],  \\
(D\de)_2(m) & \text{ if }m\in[1/k^2,\infty),  
\end{aligned} 
\right. 
\end{equation*}
where 
\begin{equation*}
\begin{aligned}
(D\de)_1(m)&:=
	 k^2 m(1-m) -k^5 m^2 \vpi\Big(\frac{m-1}{k m}\Big), \\ 
(D\de)_2(m)&:= (k^2-1) (k^2 m-1+m)-k^5 m^2 \vpi \Big(\frac{m-1}{k m}\Big),  
\end{aligned} 	
\end{equation*}
and $\vpi$ is the standard normal density function. 
Next, for $m\in(0,1/k^2]$ one has $m(1-m)>0$  and 
\begin{equation*}
	\frac{\dd}{\dd m}\Big(\frac{(D\de)_1(m)}{m(1-m)}\Big)=-\frac{k^3 \left(k^2 m^2+(m-1)^2\right)}{(1-m)^2m^2}\,\vpi \Big(\frac{m-1}{k m}\Big)<0; 
\end{equation*}
so, $(D\de)_1$ --- and hence $\de'$ --- may change in sign on the interval $(0,1/k^2]$ at most once, and only from $+$ to $-$. 
Similarly, for $m\in(1/k^2,\infty)$ one has $k^2 m-1+m>k^2 m-1>0$  and 
\begin{equation*}
	\frac{\dd}{\dd m}\Big(\frac{(D\de)_2(m)}{k^2 m-1+m}\Big)=
	-\frac{k^3 \left(k^2 m-1\right) \left(k^2 m^2+(m-1)^2\right)}{m \left(k^2 m-1+m\right)^2}\,\vpi \Big(\frac{m-1}{k m}\Big)<0;  
\end{equation*}
so, $(D\de)_2$ --- and hence $\de'$ --- may change in sign on the interval $[1/k^2,\infty)$ at most once, and only from $+$ to $-$. 
Thus, %since $\de'$ is continuous, 
$\de'$ may change in sign on the interval $(0,\infty)$ at most once, and only from $+$ to $-$. 
It follows that $\de(m)\ge\de(0+)\wedge\de(\infty-)$ for all $m\in(0,\infty)$. 
So, to complete the proof of Theorem~\ref{th:}, it remains to check that $\de(0+)\wedge\de(\infty-)\ge0$. 
In fact, one can see that $\de(0+)=0$ and 
\begin{equation}\label{eq:de(infty)}
	2\de(\infty-)=q(t):=\P(Z>t)-\frac{t\,\vpi(t)}{t^2+1}>0,  
\end{equation}
with $t:=1/k>0$. The inequality in \eqref{eq:de(infty)} is well known; see e.g.\ 
\cite[(19) for $\phi_2$]{shenton}; 
%\cite[(1.8) and Remark~2.1]{pade}; 
alternatively, it follows because $q'(t)=-\frac{2 \vpi (t)}{\left(t^2+1\right)^2}<0$ and $q(\infty-)=0$.  
This completes the entire proof of Theorem~\ref{th:}, modulo Lemma~\ref{lem:schur}. 
\end{proof}

\begin{proof}[Proof of Lemma~\ref{lem:schur}]
W.l.o.g.\ $n=2$ --- cf. e.g.\ the first equality in the big display following \eqref{eq:tp}. 
Also, by the symmetry under permutations, w.l.o.g.\ $t\in[0,\frac12]$. 
Moreover, 
w.l.o.g.\ $t\ne\frac12$; here and elsewhere we are using (sometimes tacitly) a version of continuity relevant in a given context. 
So, it suffices to show that $G'(t)\ge0$ for all $t\in[0,\frac12)$, where 
\begin{equation}\label{eq:G}
	G(t):=G_{P_1,P_2}(t):=F_{2,f}(P_{1+t},P_{2-t}). 
\end{equation}
Actually, it is enough to show that 
\begin{equation}\label{eq:G'(0)>0}
G'(0)\overset{\text{(?)}}\ge0,	
\end{equation}
because for any $\tau\in[0,\frac12)$ and $s:=\frac{t-\tau}{1-2\tau}$, one has $P_{1+t}=(1-s)P_{1+\tau}+sP_{2-\tau}$ and  $P_{2-t}=sP_{1+\tau}+(1-s)P_{2-\tau}$, whence $G_{P_1,P_2}(t)=G_{P_{1+\tau},P_{2-\tau}}(s)$ and 
$G'_{P_1,P_2}(\tau)=G'_{P_{1+\tau},P_{2-\tau}}(0)/(1-2\tau)$. 
Next --- cf.\ the proof of part (I) of Theorem~\ref{th:} --- w.l.o.g.\ $f(x)=(w-x)_+^3$ for some $w\in\R$ and all $x\in\R$.  
Thus, 
\begin{equation*}
	G(t)=\E(w-Y^{m_{1+t},s_{1+t}}-Y^{m_{2-t},s_{2-t}})_+^3, 
\end{equation*}
where $(m_u,s_u):=P_u$ for any $u$. 
If $w\le0$ then $G(t)=0$ for all $t$, so that there is nothing to prove. 
Therefore, by rescaling, w.l.o.g.\ $w=1$. 
So, in view of Definition~\ref{def:Y}, $G(t)$ can be expressed in terms of the variables $t$, $a$, $p$, $b$, $q$ only, where 
% $a_j:=\frac{s_{j+1}}{m_{j+1}}>0$ and $p_j:=\frac{m_{j+1}^2}{s_{j+1}}\in(0,1]$ for $j\in\{0,1\}$. %; w.l.o.g., $p_j\ne1$. 
\begin{equation}\label{eq:a,b,p,q:=}
	a:=\frac{s_1}{m_1}>0, \quad b:=\frac{s_2}{m_2}>0, \quad 
	p:=\frac{m_1^2}{s_1}\in(0,1], \quad q:=\frac{m_2^2}{s_2}\in(0,1].  
\end{equation}
By the symmetry relation 
$G_{P_1,P_2}(t)=G_{P_2,P_1}(t)$ (and continuity), w.l.o.g.\ $0<b<a$, so that $0<b<a<a+b$. 
Thus, it suffices to consider the following four cases: 
\begin{description}
	\item[$(C_0)$] $1\in(a+b,\infty)$; 
	\item[$(C_1)$] $1\in(a,a+b)$; 
	\item[$(C_2)$] $1\in(b,a)$; 
	\item[$(C_3)$] $1\in(0,b)$; 
\end{description}
%$(C0)\colon 1\in(a+b,\infty)$; $(C1)\colon 1\in(a,a+b)$; $(C2)\colon 1\in(b,a)$; and $(C3)\colon 1\in(0,b)$; 
at that, with each case it is assumed $0<b<a$ and $0<p,q<1$.  
In each of the cases $(C_k)$ with $k\in\{0,1,2,3\}$, the expression 
\begin{equation}\label{eq:D_k}
	D_k:=A_kG'(0), 
\end{equation}
is a polynomial in $a,b,p,q$,  
%(of degrees $3,3,2,2$ for $k=0$; $7,7,2,2$ for $k=1$; $4,5,2,2$ for $k=2$; and $4,4,2,2$ for $k=3$), 
where 
\begin{equation}\label{eq:A_k}
	\text{$A_0:=1$, $A_1:=A_3:=a^2b^2$, and $A_2:=a^2$.} 
\end{equation}
Therefore, to finish the proof of inequality \eqref{eq:G'(0)>0} and thus that of Lemma~\ref{lem:schur}, it remains to verify the following lemma. 
\end{proof}

\begin{lemma}\label{lem:G'(0)>0}
In each of the cases $(C_k)$ with $k\in\{0,1,2,3\}$, the polynomial $D_k$ in $a,b,p,q$, defined by \eqref{eq:D_k} and \eqref{eq:A_k}, is nonnegative for all $p$ and $q$ in $(0,1)$. 
\end{lemma}

\begin{proof}[Proof of Lemma~\ref{lem:G'(0)>0}] 
For each $k$, $D_k$ is a polynomial and the conditions that define the case $C_k$ are polynomial (in fact, affine) inequalities. So, the verification that $D_k$ is nonnegative in each of the cases $C_k$ can be done in a completely algorithmic manner, due to the well-known    
Tarski theory \cite{tarski48,loja,collins98}. This theory is implemented in Mathematica via \texttt{Reduce} and other related commands.  
Thus, the Mathematica command \verb9Reduce[der0 < 0 && case0]9 \big(where \verb9der09 and \verb9case09 stand for $D_0$ and \big[$(C_0)\ \&\ 0<b<a\ \&\ p\in(0,1)\ \&\ q\in(0,1)$\big], respectively\big) outputs \verb9False9 (in about $0.3$ sec on a standard desktop), which means that indeed $D_0\ge0$ in the case $(C_0)$. 
Cases $(C_1)$, $(C_2)$, and $(C_3)$ can be treated quite similarly, with Mathematica execution times of about $5.4$ sec, $0.65$ sec, and $0.04$ sec, respectively. 
%IP
Details of the corresponding calculations can be found in the Mathematica notebook solution-tarsky.nb and its pdf copy solution-tarsky.pdf in the folder Mathematica in the zip file 
LeftTailBounds.zip posted at the SelectedWorks site    %\url{http://works.bepress.com/iosif-pinelis/
\url{works.bepress.com/iosif-pinelis/7/download/}.
%\begin{center}
%\framebox{
%\parbox{4.5in}
%{
%All the .nb and .pdf files mentioned here and in the sequel are in the folder Mathematica in the zip file 
%LeftTailBounds.zip at the SelectedWorks site    \url{%http://works.bepress.com/iosif-pinelis/
%works.bepress.com/iosif-pinelis/7/download/}.   
%}
%}	
%\end{center}

The symbols $\text{\tt{der0}, \dots, \tt{der3}}$ in our Mathematica notebooks correspond to $D_0,\dots,D_3$ defined by formulas \eqref{eq:D_k}--\eqref{eq:A_k} in the paper. 
%However, such a straightforward verification takes about $15$ min for the case $(C_2)$. 
%So, in this case Mathematica can be helped as follows. It takes about $0.2$ sec for Mathematica to verify that the second (partial) derivative of $D_2$ in $a$ is nonnegative \big(in the case $(C_2)$\big), and another $0.1$ sec to verify each of the statements that $D_2$ itself and its first derivative in $a$ at $a=1$ are nonnegative as well. 

This completes the proof of Lemma~\ref{lem:G'(0)>0}, which appears no less reliable than computations done ``by hand''; cf.\ e.g.\ the views of Okounkov \cite[page~35]{okounkov-interview}, Voevodsky \cite{voevodsky-SciAmer}, and Odlyzko \cite{odlyzko-review} on computer-assisted proofs.  
\end{proof}

However, as Okounkov \cite{okounkov-interview} notes in his interview, ``perhaps we should not
be dependent on commercial software here''. Indeed, details of the execution of the Mathematica command \verb9Reduce[]9 are not open to examination. 
Therefore, in addition to the above proof, in each of the next two sections an alternative proof of Lemma~\ref{lem:G'(0)>0} is provided. 

The proof of Lemma~\ref{lem:G'(0)>0} given in Section~\ref{alt1} relies, instead of the Mathematica command \texttt{Reduce}, on the Redlog package of the computer algebra system Reduce; both Reduce and Redlog are open-source and freely distributed (\url{http://www.redlog.eu/}). 

The proof of Lemma~\ref{lem:G'(0)>0} given in Section~\ref{alt2} uses only standard tools of calculus and also such a standard tool of algebra as the resultant, available in a number of open-source computer algebra software packages. 

Recall that, for each $k\in\{0,1,2,3\}$, 
$D_k$
is a polynomial in 
$a,b,p,q$. 
For each $k\in\{0,1,2,3\}$, in the case $(C_k)$, the quadruple $(a,b,p,q)$ belongs to the set %$\om_k\times[0,1]^2$,  
\begin{equation}\label{eq:Om}
	\Om_k:=\om_k\times(0,1)^2, 
\end{equation}
where 
\begin{equation}\label{eq:om}
\begin{aligned}
	\om_0:=&\{(a,b)\in\R^2\colon 0<b<a<a + b<1\} \\ 
	&=\{(a,b)\in\R^2\colon b<a<a + b<1\}, \\ 
	\om_1:=&\{(a,b)\in\R^2\colon 0<b<a < 1 < a + b\} \\ 
					&=\{(a,b)\in\R^2\colon b<a < 1 < a + b\}, \\ 	
	\om_2:=&\{(a,b)\in\R^2\colon 0<b<1<a\}, \\ 	
	\om_3:=&\{(a,b)\in\R^2\colon 1<b<a\}.  	
\end{aligned}	
\end{equation}
For each $k\in\{0,1,2,3\}$, let $\bar\om_k$ denote the topological closure of $\om_k$, so that $\bar\om_k$ is defined by the system of non-strict inequalities corresponding to the strict inequalities defining the set $\om_k$. 

We shall use notation such as the following: 
\begin{equation}\label{eq:D_;}
	D_{k;p=\de}:=D_k\big|_{p=\de},\quad D_{k;q=\vp}:=D_k\big|_{q=\vp},\quad D_{k;p=\de,q=\vp}:=D_k\big|_{p=\de,q=\vp};  
\end{equation}
sometimes in such notation we shall use, instead of $D_k$, a modified version $\tD_k$ of $D_k$, which differs from $D_k$ by a factor which is manifestly positive in the corresponding context. 

\begin{center}
\framebox{
\parbox{3.5in}
{
In files pertaining to the mentioned package Redlog and in subsequently used Mathematica notebooks, we shall use notations 
such as $\text{\tt{der0}},\dots,\text{\tt{der3}},\text{\tt{der0p0}},\dots,\text{\tt{der3p1q1}}$ for $D_0,\dots,D_3,D_{0;p=0},\dots,D_{3;p=1,q=1}$, respectively, or possibly for $\tD_\cdot$ in place of $D_\cdot$. 
}
}	
\end{center}
 
\section{First alternative proof of \texorpdfstring{Lemma~\ref{lem:G'(0)>0}}{L3.2}}\label{alt1}
Unfortunately, for polynomials in several variables the mentioned package Redlog is either much slower than Mathematica (as in the cases of the polynomials $D_0$ and $D_3$ in \eqref{eq:D_k}) or unable to complete the desired verification of the nonnegativity (as in the cases of the polynomials $D_1$ and $D_2$ in \eqref{eq:D_k}%
%C:\Users\ipinelis\Dropbox\mtu\PU-left-tail-improved\pin-utev-better\paper\arxiv\arxiv_v3\Reduce(redlog)\der1q1-crash.log, 
%C:\Users\ipinelis\Dropbox\mtu\PU-left-tail-improved\pin-utev-better\paper\arxiv\arxiv_v3\Reduce(redlog)\der1q1-crash.PNG 
). I have also tried another well-known open-source program, QEPCAD B (Quantifier Elimination by Partial Cylindrical Algebraic Decomposition, Version B), but it crashes % der0p0-qepcad-crash.txt
even where Redlog eventually produces the result. 

%!!!!!! refs. to corr. files \\ 
More specifically, Redlog verifies the nonnegativity of the polynomials $D_0$ and $D_3$ (in cases $(C_0)$ and $(C_3)$) in about $107$ min and $0.45$ sec, respectively; details on this can be found in the .log files der0.log and der3.log 
and in the corresponding .png files der0.png and der3.png.  
\begin{center}
\framebox{
\parbox{4.5in}
{
The .log and .png files mentioned in this section are in the folder \break Reduce(Redlog) in the zip file 
LeftTailBounds.zip at the SelectedWorks site  \url{%http://works.bepress.com/iosif-pinelis/
works.bepress.com/iosif-pinelis/7/download/}.   
}
}	
\end{center}
 
These execution times, $107$ min and $0.45$ sec, may be compared with the corresponding ones for Mathematica, mentioned in the proof of Lemma~\ref{lem:G'(0)>0} in the preceding section: $0.3$ sec and $0.04$ sec). 
%Just as in our Mathematica notebooks, symbols $\text{\tt{der0}, \dots, \tt{der3}}$ in our Redlog files correspond to $D_0,\dots,D_3$ defined by formulas \eqref{eq:D_k}--\eqref{eq:A_k} in the paper. 

To verify the nonnegativity of the polynomials $D_1$ and $D_2$ with Redlog, each of these two verification problems has to be reduced, by a human, to a series (or rather a tree) of simpler problems, as presented below.  

\begin{lemma}\label{lem:D1}
In the case $(C_1)$, the polynomial $D_1$ in $a,b,p,q$ is nonnegative for all $p$ and $q$ in $(0,1)$ -- that is, $D_1\ge0$ for all $(a,b,p,q)\in\Om_1$.
\end{lemma}

\begin{proof}
Assume indeed in this proof that $(a,b,p,q)\in\Om_1$, unless otherwise stated. 
One has 
\begin{multline}\label{eq:D_1}
	D_1=a^5 b^2 p - 2 a^4 b^3 p + a^3 b^4 p + a^4 p^2 - 3 a^5 p^2 + 
 3 a^6 p^2 - a^7 p^2 \\ 
 - 2 a^3 b p^2 + 6 a^4 b p^2 - 6 a^5 b p^2 + 
 2 a^6 b p^2 + 3 a^3 b^2 p^2 - 3 a^4 b^2 p^2 \\ 
 + 2 a^4 b^3 p^2 - 
 a^3 b^4 p^2 + a^4 b^3 q - 2 a^3 b^4 q + a^2 b^5 q + 2 a^2 b^2 p q \\ 
 - 6 a^3 b^2 p q + 6 a^4 b^2 p q - 2 a^5 b^2 p q - 6 a^2 b^3 p q + 
 6 a^2 b^4 p q - 2 a^2 b^5 p q \\ 
 - 2 a b^3 q^2 + 3 a^2 b^3 q^2 - 
 a^4 b^3 q^2 + b^4 q^2 + 6 a b^4 q^2 - 3 a^2 b^4 q^2 + 
 2 a^3 b^4 q^2 \\ 
 - 3 b^5 q^2 - 6 a b^5 q^2 + 3 b^6 q^2 + 2 a b^6 q^2 - 
 b^7 q^2. 
\end{multline} 
Consider  
\begin{multline}\label{eq:det1}
	\frac{\pd p^2 D_1\,\pd q^2 D_1-(\pd p\pd q D_1)^2}{4 a^3 (a - b)^2 b^3} = \Det_1 \\
	:=-2 + 9 a - 15 a^2 + 10 a^3 - 3 a^5 + a^6 + 9 b - 33 a b + 48 a^2 b - 
 36 a^3 b + 15 a^4 b \\ 
 - 3 a^5 b - 15 b^2 + 48 a b^2 - 54 a^2 b^2 + 
 24 a^3 b^2 - 3 a^4 b^2 + 10 b^3 - 36 a b^3 + 24 a^2 b^3 \\ 
 - 7 a^3 b^3 + 15 a b^4 - 3 a^2 b^4 - 3 b^5 - 3 a b^5 + b^6;  	
\end{multline} 
here and in the sequel, $\pd\al$ denotes, as usual, the partial differentiation in $\al$. 
Using the mentioned package Redlog, we see that $\Det_1<0$ on $\om_1$; this takes about $0.5$ sec; see details in the  files der1det.log 
and 
der1det.png.  

Hence, the determinant of the Hessian matrix of $D_1$ with respect to $p$ and $q$ is negative for all $(a,b,p,q)\in\Om_1$. 
It follows that $D_1$ is saddle-like in $p$ and $q$, and so, for each fixed $(a,b)\in\om_1$, the minimum of the polynomial $D_1$ in $(p,q)\in[0,1]^2$ is not attained at any point $(p,q)\in(0,1)^2$; therefore, this minimum is attained at some point $(p,q)$ on the boundary of the unit square $[0,1]^2$. 

Consider then each of the four boundary subcases of Case~1: $p=0$, $p=1$, $q=0$, and $q=1$. 
Using Redlog, we see that 
$D_{1;p=0}\ge0$ for $(a,b,q)\in\om_1\times(0,1)$ (execution time $\approx 0.25$ sec; details in the files der1p0.log 
and der1p0.png) and 
$D_{1;q=0}\ge0$ for $(a,b,p)\in\om_1\times(0,1)$ (execution time $\approx 0.25$ sec; details in the files der1q0.log 
and der1q0.png). 

The subcases $p=1$ and $q=1$ require more care. Recall notation \eqref{eq:D_;}. 

To consider the subcase $p=1$, 
assume that $(a,b)\in\om_1$ and $q\in(0,1)$. 
In view of \eqref{eq:D_1}, 
\begin{multline}\label{eq:D1;p=1}
D_{1;p=1}=a^3 (a - 3 a^2 + 3 a^3 - a^4 - 2 b + 6 a b - 6 a^2 b + 2 a^3 b + 3 b^2 - 
 3 a b^2 + a^2 b^2) \\
 + 
 a^2 b^2 (2 - 6 a + 6 a^2 - 2 a^3 - 6 b + a^2 b + 6 b^2 - 2 a b^2 - b^3) q \\
 - 
 b^3 (2 a - 3 a^2 + a^4 - b - 6 a b + 3 a^2 b - 2 a^3 b + 3 b^2 + 
    6 a b^2 - 3 b^3 - 2 a b^3 + b^4) q^2.  	
\end{multline}
Using Redlog, we see (in about $0.16$ sec) that  
\begin{equation}\label{eq:D002D1;p=1}
	D^{002}_{1;p=1} := \frac{\pd q^2 D_{1;p=1} }{2 b^3}
	=3 a^2 (1 - b) - (2 a - b) (1 - b)^3 + a^3 (2  b - a)\ge0 
\end{equation}
and (in about $0.8$ sec) that 
\begin{equation}\label{eq:D001D1;p=1}
\begin{gathered}
%D_{1;p=1}\big|_{q=1}=(a - b)^2[a^2 (1 - a)^3 + (1 - b)^3 b^2]>0, \\ 
		D^{001}_{1;p=1,q=1} := \frac{\pd q D_{1;p=1}\big|_{q=1} }{b^2} \qquad\qquad\qquad\qquad\qquad\qquad\qquad\qquad\qquad\qquad \\
	=a^4 (6 - b - 2 a) - 2 (1 - b)^3 b (2 a - b) - 2 a^3 (3 - b^2) + 
 a^2 (2 - b^3)\le0   
\end{gathered}	
\end{equation}
(details in the files der1p1.log 
and der1p1.png; notations \tt{D002der1p1} and 
\break 
\tt{D001der1p1q1} there correspond to $D^{002}_{1;p=1}$ and $D^{001}_{1;p=1,q=1}$, respectively). 
So, $D_{1;p=1}$ is convex and decreasing in $q$. At that, 
\begin{equation}\label{eq:D1;p=1,q=1}
	D_{1;p=1}\big|_{q=1}=(a - b)^2[a^2 (1 - a)^3 + (1 - b)^3 b^2]>0.  
\end{equation}
We conclude that indeed $D_{1;p=1}\ge0$. 

To complete the proof of Lemma~\ref{lem:D1}, it remains to consider the subcase $q=1$. % proof_of_Lem3.2\der1q1-new.nb
Expanding $D_{1;q=1}$ in powers of $p$, one has 
\begin{equation*}
	D_{1;q=1}=\psi(p):=A p^2 + B p + C, 
\end{equation*}
where 
\begin{align}
	A&:=-a^3 (a^4-2 a^3 b-3 a^3+6 a^2 b+3 a^2-2 a b^3+3 a b^2-6 a b-a+b^4 \notag \\ 
	&\qquad\qquad\qquad\qquad\qquad\qquad\qquad\qquad\qquad\qquad\qquad\qquad -3 b^2+2 b), \label{eq:A:-} \\ 
B&:=-a^2 b^2 (a^3+2 a^2 b-6 a^2-a b^2+6 a+2 b^3-6 b^2+6 b-2), \notag \\ 
	C&:=b^3 (a^2 b^2-3 a^2 b+3 a^2+2 a b^3-6 a b^2+6 a b-2 a-b^4+3 b^3-3 b^2+b). \notag 
\end{align}
Let 
\begin{align*}
	\text{discr}:=&\frac{B^2 - 4 A C}{a^3 b^3 (a - b)^2} \\ 
	=&8 - 36 a + 60 a^2 - 44 a^3 + 12 a^4 - 36 b + 132 a b - 180 a^2 b + 
 108 a^3 b - 24 a^4 b \\
 &+ a^5 b + 60 b^2 - 192 a b^2 + 192 a^2 b^2 - 84 a^3 b^2 + 10 a^4 b^2 - 40 b^3 + 144 a b^3 \\ 
 &- 84 a^2 b^3 + 
 21 a^3 b^3 - 60 a b^4 + 12 a^2 b^4 + 12 b^5 + 12 a b^5 - 4 b^6, \\ 
 d_1:=&\frac{\psi'(0)}{a^2 b^2}=\frac B{a^2 b^2}=2 - 6 a + 6 a^2 - a^3 - 6 b - 2 a^2 b + 6 b^2 + a b^2 - 2 b^3, \\ 
 d_2:=&\frac{\psi'(1)}{a^2 (a-b)}=\frac{2A+B}{a^2 (a-b)}  
 =2 a - 6 a^2 + 6 a^3 - 2 a^4 - 2 b + 6 a b - 6 a^2 b + 2 a^3 b \\
 &\qquad\qquad\qquad\qquad\qquad\qquad\qquad + 6 b^2 - 6 a b^2 + a^2 b^2 - 6 b^3 + 3 a b^3 + 2 b^4.  
\end{align*}
Note that discr equals in sign the discriminant of the quadratic polynomial $\psi(p)$. Therefore, $\text{discr}>0$ if and only if $\psi(p)$ takes both positive and negative values as $p$ varies from $-\infty$ to $\infty$. 
Using Redlog, we see that (i) $A\ge0$ ($\approx 0.16$ sec execution time); 
(ii) the conjunction of the conditions $\text{discr}>0$, $d_1<0$, and $b<1/2$ never takes place over the set $\om_1$ ($\approx 3.1$ sec execution time); and (iii) the conjunction of the conditions $\text{discr}>0$, $d_2>0$, and $b>1/2$ never takes place over the set $\om_1$ ($\approx 25.5$ min execution time); details are in the files der1q1.log, der1q1-top.png (for the first 10 Redlog commands), and der1q1-bottom.png (for the last 10 Redlog commands); in those files, \tt{AA} stands for $A/a^3$, with $A$ as in \eqref{eq:A:-}. 
So, over the set $\om_1$ one has the following: (i') the function $\psi$ is convex; (ii') if $b<1/2$ and $\psi$ changes sign over $\R$, then $\psi'(0)\ge0$ and hence $\psi(p)$ is nondecreasing in $p\in[0,1]$; and (iii') if $b>1/2$ and $\psi$ changes sign over $\R$, then $\psi'(1)\le0$ and hence $\psi(p)$ is nonincreasing in $p\in[0,1]$. 
Thus, in view of the continuity of $D_{1;q=1}=\psi(p)$ in $b$, it remains to verify that $\tD_{1;q=1,p=0}:=\psi(0)/b^3=C/b^3$ and $\tD_{1;q=1,p=1}=\psi(1)/(a - b)^2=(A+B+C)/(a - b)^2$ are both nonnegative (over $\om_1$). For $\tD_{1;q=1,p=0}$ this is checked by Redlog in about $0.05$ sec (details in files der1q1.log and der1q1-bottom.png), whereas $\tD_{1;q=1,p=1}=a^2 (1 - a)^3 + b^2 (1 - b)^3$ is manifestly positive (over $\om_1$). 

This completes the proof of Lemma~\ref{lem:D1}. 
\end{proof}

%Note that $D_0$ is saddle-like in $p$ and $q$, that is, 
%the determinant 
%\begin{equation*}
%	\pd p^2 D_0\,\pd q^2 D_0-(\pd p\pd q D_0)^2
%	=-36 a^2 (a - b)^2 b^2
%\end{equation*}
%of the Hessian matrix of $D_0$ with respect to $p$ and $q$ is negative for all $(a,b,p,q)\in\Om_0$; here and in the sequel, $\pd\al$ denotes, as usual, the partial differentiation in $\al$. Hence, the case $k=0$ of Lemma~\ref{lem:0<p,q<1} follows. 

% C:\Users\ipinelis\Dropbox\mtu\PU-left-tail-improved\pin-utev-better\paper\arxiv\arxiv_v3\qepcad-for-PULeft\der0p0-qepcad-crash.txt 

\begin{lemma}\label{lem:D2}
In the case $(C_2)$, the polynomial $D_2$ in $a,b,p,q$ is nonnegative for all $p$ and $q$ in $(0,1)$ -- that is, $D_2\ge0$ for all $(a,b,p,q)\in\Om_1$.
\end{lemma}

\begin{proof}
Assume indeed in this proof that $(a,b,p,q)\in\Om_2$, unless otherwise stated. 
% see \proof_of_Lem3.2\solution-tarsky.nb and, in particular, (* faster way: *) there
One has 
\begin{multline}\label{eq:D_2}
	D_2=a^2 p - 3 a^3 p + 3 a^4 p - 2 a^4 b p + a^3 b^2 p + 3 a^3 p^2 - 
 3 a^4 p^2 + 2 a^4 b p^2 - a^3 b^2 p^2 - 2 a b q \\ 
 + 3 a^2 b q + 
 b^2 q - 3 a^2 b^2 q + a^2 b^3 q - 6 a^2 b p q + 6 a^2 b^2 p q - 
 2 a^2 b^3 p q + 6 a b^2 q^2 - 3 b^3 q^2 - 6 a b^3 q^2 \\ 
 + 3 b^4 q^2 + 
 2 a b^4 q^2 - b^5 q^2. 
\end{multline} 
Using Redlog (details in the files der2.log, der2-top.png, and der2-bottom.png),   
we see that  
\begin{multline}\label{eq:der2DDa}
	\tfrac12\,\pd a^2 D_2
	=p - 9 a p + 18 a^2 p - 12 a^2 b p + 3 a b^2 p + 9 a p^2 - 
 18 a^2 p^2 + 12 a^2 b p^2 - 3 a b^2 p^2 \\
 + 3 b q - 3 b^2 q + b^3 q - 
 6 b p q + 6 b^2 p q - 2 b^3 p q\ge0  	
\end{multline} 
on $\Om_2$ (execution time $\approx22.6$ min) -- so that $D_2$ is convex in $a$, 
\begin{multline}\label{eq:der2a1}
	D_{2;a=1}
	=p - 2 b p + b^2 p + 2 b p^2 - b^2 p^2 + b q - 2 b^2 q + b^3 q - 
 6 b p q + 6 b^2 p q - 2 b^3 p q + 6 b^2 q^2 \\ 
 - 9 b^3 q^2 + 5 b^4 q^2 -
  b^5 q^2\ge0  	
\end{multline} 
for $b,p,q$ in $(0,1)$ (execution time $\approx1.2$ sec), and 
\begin{multline}\label{eq:Dder2a1}
	\pd a D_2\big|_{a=1}
	=5 p - 8 b p + 3 b^2 p - 3 p^2 + 8 b p^2 - 3 b^2 p^2 + 4 b q - 
 6 b^2 q + 2 b^3 q - 12 b p q + 12 b^2 p q \\
 - 4 b^3 p q + 6 b^2 q^2 - 
 6 b^3 q^2 + 2 b^4 q^2\ge0  	
\end{multline} 
for $b,p,q$ in $(0,1)$ (execution time $\approx1.2$ sec);   
the symbols \tt{der2DDa}, \tt{der2a1}, and \tt{Dder2a1} in the mentioned Redlog files stand for $\tfrac12\,\pd a^2 D_2$, $D_{2;a=1}$, and $\pd a D_2\big|_{a=1}$, respectively.     
To complete the proof of Lemma~\ref{lem:D2}, it remains to recall the definition \eqref{eq:Om}. 
\end{proof}

Lemma~\ref{lem:G'(0)>0} follows immediately from Lemmas~\ref{lem:D1} and \ref{lem:D2} and the nonnegativity of $D_0$ and $D_3$, mentioned in the beginning of this section.

\section{Second alternative proof of \texorpdfstring{Lemma~\ref{lem:G'(0)>0}}{L3.2}}\label{alt2}
First here, let us briefly describe how to use the resultant tool in problems of polynomial optimization. Let $K[x_1,\dots,x_n]$ denote the ring of all polynomials in indeterminates $x_1,\dots,x_n$ over a field $K$; see e.g.\ \cite{vanderWaerden} for the algebraic terminology used in this description. 

In the case when $n=1$, this ring is written as $K[x]$, the ring of all polynomials in $x$ over $K$. Suppose now that $f(x)=a_0x^n+\dots+a_n$ and $g(x)=b_0x^n+\dots+b_n$ are two polynomials in $K[x]$. %, with $a_0b_0\ne0$. 
The resultant $R(f,g)=R_K(f(x),g(x))$ of these two polynomials is the determinant of the %\break  
$(n+m)\times(n+m)$ Sylvester matrix 
%$\big(a_{j-i} \ii{i\leq m\ \&\ 0\leq j-i\leq n} \break 
%+b(j-i+m) \ii{i>m\ \&\ 0\leq j-i+m\leq m}\colon i\in\{1,\dots,n+m\},j\in\{1,\dots,n+m\}\big)$. 
%\\
%$\big(a_{j-i} \ii{i\leq m} 
%+b(j-i+m) \ii{i>m}\colon i,j\text{ in }\{1,\dots,n+m\}\big)$. 
%\\
$\big(a_{j-i} \ii{i\leq m} 
+b_{j-i+m} \ii{i>m}\big)_{i,j=1}^{n+m}$; here it is assumed that $a_k=0$ if \break 
$k\notin%IP04.13.15 \int
\intr0n$ and $b_\ell=0$ if $\ell\notin\intr0m$. 
Thus, $R(f,g)$ is a homogeneous polynomial of degree $n+m$ in %the coefficients 
$a_0,\dots,a_n,b_0,\dots,b_m$. 
The Mathematica notebook resultant.nb and its pdf copy resultant.pdf show the Sylvester matrix for $n=3$ and $m=4$.    
\begin{center}
\framebox{
\parbox{4.5in}
{
The .nb files and their .pdf copies mentioned in this section are in the folder Mathematica in the zip file 
LeftTailBounds.zip at the SelectedWorks site   \url{%http://works.bepress.com/iosif-pinelis/
works.bepress.com/iosif-pinelis/7/download/}.   
}
}	
\end{center}

The remarkable property of the resultant is that $R(f,g)=0$ if and only if $a_0=b_0=0$ or the polynomials $f(x)$ and $g(x)$ have a common root, possibly in an algebraically closed field $C$ containing the field $K$; 
moreover, 
\begin{equation}\label{eq:R=}
	R(f,g)=A(x)f(x)+B(x)g(x)
\end{equation}
for some polynomials $A(x)$ and $B(x)$ in $K[x]$, whose coefficients are polynomials %IP04.13.15
over $\Z$ in $a_0,\dots,a_n,b_0,\dots,b_m$;  
see e.g.\ \cite[Section~5.8]{vanderWaerden}. 

Take now any natural $n$ and let $f(x_1,\dots,x_n)$ and $g(x_1,\dots,x_n)$ be any polynomials in $K[x_1,\dots,x_n]$. These polynomials may be identified with the corresponding polynomials $f(x_1,\dots,x_{n-1})(x_n)$ and $g(x_1,\dots,x_{n-1})(x_n)$ in the ring $K(x_1,\dots,x_{n-1})[x_n]$ of all polynomials in the single indeterminate $x_n$ over the field $K(x_1,\dots,x_{n-1})$ of all rational functions in indeterminates $x_1,\dots,x_{n-1}$ over the field $K$. Thus, one has the resultant 
\begin{multline*}
	R_{x_n}(f,g)(x_1,\dots,x_{n-1})\\ 
	:=R_{K(x_1,\dots,x_{n-1})}\big(f(x_1,\dots,x_{n-1})(x_n),\,g(x_1,\dots,x_{n-1})(x_n)\big) 
\end{multline*}
of the polynomials $f(x_1,\dots,x_n)$ and $g(x_1,\dots,x_n)$ with respect to indeterminate $x_n$. 
Clearly, $R_{x_n}(f,g)(x_1,\dots,x_{n-1})\in K[x_1,\dots,x_{n-1}]$. 
Moreover, by \eqref{eq:R=}, 
\begin{equation}\label{eq:R_n=}
	R_{x_n}(f,g)(x_1,\dots,x_{n-1})=A(x_1,\dots,x_n)f(x_1,\dots,x_n)+B(x_1,\dots,x_n)g(x_1,\dots,x_n)
\end{equation} 
for some polynomials $A(x_1,\dots,x_n)$ and $B(x_1,\dots,x_n)$ in $K[x_1,\dots,x_n]$. 
So, if the polynomials $f(x_1,\dots,x_n)$ and $g(x_1,\dots,x_n)$ have a common root $(\al_1,\dots,\al_n)\in C^n$, then $R_{x_n}(f,g)(x_1,\dots,x_{n-1})$ has a root -- namely, $(\al_1,\dots,\al_{n-1})$ -- in $C^{n-1}$. 

Consider now a system of $n$ polynomial equations 
\begin{equation*}
\text{$f_j(x_1,\dots,x_n)=0$ for all $j\in\intr1n$, }	
\end{equation*}
in the $n$ indeterminates $x_1,\dots,x_n$, where $f_j(x_1,\dots,x_n)\in K[x_1,\dots,x_n]$ for each $j$. If this system has a root $(\al_1,\dots,\al_n)\in C^n$, then the reduced system of the $n-1$ polynomial equations 
\begin{equation}\label{eq:R_n}
\text{$R_{x_n}(f_j,f_{j+1})(x_1,\dots,x_{n-1})=0$ for  $j\in\intr1{n-1}$, }	
\end{equation}
in the $n-1$ indeterminates $x_1,\dots,x_{n-1}$, has a root -- namely, $(\al_1,\dots,\al_{n-1})$ -- in $C^{n-1}$. Thus, the indeterminate $x_n$ has been eliminated. Continuing in this manner, one arrives at one equation of the form $F_1(x_1)=0$, for some polynomial $F_1(x_1)\in K[x_1]$. Quite similarly one obtains polynomial equations $F_j(x_j)=0$ for each $j\in\intr2n$. In the ``nondegenerate'' case -- when all the polynomials $F_1(x_1),\dots,F_n(x_n)$ are nonzero -- each of the resulting equations $F_1(x_1)=0,\dots,F_n(x_n)=0$ has only finitely many roots (in $C$ and hence in $K$). One can then check which of the finitely many $n$-tuples of those roots are roots of the original system $f_j(x_1,\dots,x_n)=0$ for all $j=1,\dots,n$. Thus, one can see that, at least in the ``nondegenerate'' case, resultants can be used to solve systems of polynomial equations by successive elimination, somewhat similarly to solving systems of linear equations. 

In the remaining, ``degenerate'' case, other, more computationally intensive tools of algebraic geometry need to be used, such as the calculation of a Gr\"obner basis, which, in particular, allows one to determine the dimension of an algebraic variety; see e.g.\ \cite[Ch.\ 9]{cox-little-oshea}. One may hope, though, that such a degeneracy is unlikely to occur in a particular problem. Also, if a degeneracy indeed occurs, one may turn to using other methods, say ones of calculus if the field $K$ is $\R$. 

Suppose now, in the case when (say) $K=\R$, one wants to show, as we do in the proof of Lemma~\ref{lem:G'(0)>0}, that a polynomial $f(x_1,\dots,x_n)$ is nonnegative everywhere on a (say) compact subset of $\R^n$ of the form 
\begin{equation*}
	\Om:=\{(x_1,\dots,x_n)\in\R^n\colon g_i(x_1,\dots,x_n)\ge0\ \forall i\in\intr1m\}, 
\end{equation*}
for some natural $m$ and some nonzero polynomials $g_i(x_1,\dots,x_n)$ in \break  $\R[x_1,\dots,x_n]$. First here, if the minimum of $f(x_1,\dots,x_n)$ over $(x_1,\dots,x_n)\in\Om$ is attained at an interior point $(\al_1,\dots,\al_n)$ of $\Om$, then this point is critical for $f$, that is, the partial derivatives, say $f_1,\dots,f_n$, of the function $f$ respectively in $x_1,\dots,x_n$ vanish at the point $(\al_1,\dots,\al_n)$. So, one can use the resultants, as described above, to obtain a finite set containing all the critical points of $f$ in the interior set 
\begin{equation*}
	\inter\Om=\{(x_1,\dots,x_n)\in\R^k\colon g_i(x_1,\dots,x_n)>0\ \forall i\in\intr1m\} 
\end{equation*}
of $\Om$ -- provided that the non-degeneracy holds. 
Similarly, for any set $J\subset\intr1m$, one can try to obtain a finite set containing all the critical points of $f$ in the interior set of the (in general curved) $J$-face 
\begin{equation*}
	%\Om_J:=
	\{(x_1,\dots,x_n)\in\Om\colon 
	g_i(x_1,\dots,x_n)=0\ \forall i\in J\}
%		\inter\Om_J=\{(x_1,\dots,x_n)\in\R^k\colon 
%	g_i(x_1,\dots,x_n)=0\ \forall i\in J,\ g_i(x_1,\dots,x_n)>0\ \forall i\in\{1,\dots,m\}\setminus J\}
	\end{equation*}
of the set $\Om$ -- by considering the polynomial system of equations involving the Lagrange multipliers; see e.g.\ \cite[page~434]{pourciau}. 
\big(If the faces of $\Om$ are subsets of affine subspaces of $\R^n$ -- as they are in the proof of Lemma~\ref{lem:G'(0)>0}, then the elimination of variables and hence the minimization of $f$ over the faces are much simpler.\big)
Thus, unless a degeneracy is encountered, one reduces the verification of the nonnegativity of the polynomial function $f$ on the set $\Om$ to that on a finite set.    

One can also try to use some of the various Positivstellens\"atze of real algebraic geometry (see e.g.\ \cite{krivine64a,krivine64b,cassier,handelman85,handelman88}), which can provide a so-called certificate of positivity to a polynomial that is indeed positive on a set defined by a system of polynomial inequalities (over $\R$); that is, by an appropriate Positivstellensatz, the positive polynomial can be represented as a polynomial (with positive coefficients) in simpler polynomials that are manifestly positive on the given set. 
This method was used successfully in \cite{pin-hoeff-published}. However, it does not appear to be very effective in the proof of of Lemma~\ref{lem:G'(0)>0} and will be used only little there. 

One can also use a combination of all these and/or other methods. 
In fact, we shall try to avoid, as much as we can, using resultants or other algebraic tools. Instead, we shall try to use, as much as possible, calculus tools such as monotonicity and concavity/convexity, which are oftentimes much more efficient in eliminating variables. 

Apparently any proof of Lemma~\ref{lem:G'(0)>0} will involve a very large amount of algebraic and arithmetic calculations; in particular, note the execution time of of about $5.4$ sec mentioned in the above proof of Lemma~\ref{lem:G'(0)>0}, based on Tarski's theory. Of course, for the proof to be valid, any arithmetic calculation needed therein must be carried out in an exact arithmetic. 

\begin{center}
	***
\end{center}

Let us now get specifically to the proof of Lemma~\ref{lem:G'(0)>0}. 
Recall definitions \eqref{eq:Om} and \eqref{eq:om}. 
%Recall that, for each $k\in\{0,1,2,3\}$, 
%$D_k$
%is a polynomial in 
%$a,b,p,q$. 
%For each $k\in\{0,1,2,3\}$, in the case $(C_k)$, the quadruple $(a,b,p,q)$ belongs to the set %$\om_k\times[0,1]^2$,  
%\begin{equation}\label{eq:Om}
%	\Om_k:=\om_k\times(0,1)^2, 
%\end{equation}
%where 
%\begin{align}\label{eq:om}
%	\om_0:=&\{(a,b)\in\R^2\colon 0<b<a<a + b<1\} 
%	=\{(a,b)\in\R^2\colon b<a<a + b<1\}, \\ 
%	\om_1:=&\{(a,b)\in\R^2\colon 0<b<a < 1 < a + b\} 
%					=\{(a,b)\in\R^2\colon b<a < 1 < a + b\}, \\ 	
%	\om_2:=&\{(a,b)\in\R^2\colon 0<b<1<a\}, \\ 	
%	\om_3:=&\{(a,b)\in\R^2\colon 1<b<a\}.  	
%\end{align}
%For each $k\in\{0,1,2,3\}$, let $\bar\om_k$ denote the topological closure of $\om_k$, so that $\bar\om_k$ is defined by the system of non-strict inequalities corresponding to the strict inequalities defining the set $\om_k$.

The following lemma eliminates one of the variables $p,q$. 

\begin{lemma}\label{lem:0<p,q<1}
For each $k\in\{0,1,2,3\}$ and for each fixed $(a,b)\in\om_k$, the minimum of the polynomial $D_k$ in $(p,q)\in[0,1]^2$ is not attained at any point $(p,q)\in(0,1)^2$; hence, this minimum is attained at some point $(p,q)$ on the boundary of the unit square $[0,1]^2$. 
%
%not attained at any point $(a,b,p,q)\in\Om_k$ and hence is attained only on the boundary of 
%$(a,b,p,q)\in\bar\Om_k$ such that $0<p,q<1$. 
\end{lemma} 

\begin{proof}[Proof of Lemma~\ref{lem:0<p,q<1}]
Details of calculations in this proof can be found in Mathematica notebook noInnerExtrIn\_pq.nb and its pdf copy %\break 
noInnerExtrIn\_pq.pdf.  
One has 
\begin{multline}\label{eq:D_0}
	D_0=a^3 p - 2 a^2 b p + a b^2 p + 6 a^2 p^2 - 6 a^3 p^2 + a^2 b q - 
 2 a b^2 q + b^3 q \\ - 12 a b p q  
 + 6 a^2 b p q + 6 a b^2 p q + 
 6 b^2 q^2 - 6 b^3 q^2. 
\end{multline} 
Note that $D_0$ is saddle-like in $p$ and $q$, that is, 
the determinant 
\begin{equation*}
	\pd p^2 D_0\,\pd q^2 D_0-(\pd p\pd q D_0)^2
	=-36 a^2 (a - b)^2 b^2
\end{equation*}
of the Hessian matrix of $D_0$ with respect to $p$ and $q$ is negative for all $(a,b,p,q)\in\Om_0$. 
%; here and in the sequel, $\pd\al$ denotes, as usual, the partial differentiation in $\al$. 
Hence, the case $k=0$ of Lemma~\ref{lem:0<p,q<1} follows. 

The other cases of Lemma~\ref{lem:0<p,q<1} are similar, except that the case $k=1$ is more complicated than the rest. Let us defer this case to the end of the proof of Lemma~\ref{lem:0<p,q<1} and consider the cases $k=2$ and $k=3$ next. 

Concerning $k=2$, recall \eqref{eq:D_2}. 
%one has 
%\begin{multline}\label{eq:D_2}
%	D_2=a^2 p - 3 a^3 p + 3 a^4 p - 2 a^4 b p + a^3 b^2 p + 3 a^3 p^2 - 
% 3 a^4 p^2 + 2 a^4 b p^2 \\ 
% - a^3 b^2 p^2  
% - 2 a b q + 3 a^2 b q + 
% b^2 q - 3 a^2 b^2 q + a^2 b^3 q - 6 a^2 b p q + 6 a^2 b^2 p q \\
%  - 2 a^2 b^3 p q 
% + 6 a b^2 q^2 - 3 b^3 q^2 - 6 a b^3 q^2 + 3 b^4 q^2 + 
% 2 a b^4 q^2 - b^5 q^2 
%\end{multline} 
%and 
One has 
\begin{multline*}
	\pd p^2 D_2\,\pd q^2 D_2-(\pd p\pd q D_2)^2 \\ 
	=-4 a^3 (a - b) b^2 (3 - 3 b + b^2) [(1 - b) (3 - b) + (a - 1) (6 - 4 b)] <0
\end{multline*}
for all $(a,b,p,q)\in\Om_2$. Hence, the case $k=2$ of Lemma~\ref{lem:0<p,q<1} follows. 

Concerning $k=3$, one has 
\begin{multline}\label{eq:D_3}
	D_3=a^4 p - 2 a^3 b p + a^2 b^2 p - a^4 p^2 + 2 a^3 b p^2 + a^2 b^2 q \\ 
 -2 a b^3 q + b^4 q - 2 a^2 b^2 p q + 2 a b^3 q^2 - b^4 q^2
\end{multline} 
and 
\begin{equation*}
	\pd p^2 D_3\,\pd q^2 D_3-(\pd p\pd q D_3)^2  
	=-8 a^3 (a - b)^2 b^3 <0	
\end{equation*}
for all $(a,b,p,q)\in\Om_3$. Hence, the case $k=3$ of Lemma~\ref{lem:0<p,q<1} follows. 

It remains to consider the case $k=1$; further details on this, more difficult case are given in the files noInnerExtrIn\_pq-scratch.nb and \break 
noInnerExtrIn\_pq-scratch.pdf. Recall \eqref{eq:D_1} and \eqref{eq:det1}.  
%\begin{multline}\label{eq:D_1}
%	D_1=a^5 b^2 p - 2 a^4 b^3 p + a^3 b^4 p + a^4 p^2 - 3 a^5 p^2 + 
% 3 a^6 p^2 - a^7 p^2 \\ 
% - 2 a^3 b p^2 + 6 a^4 b p^2 - 6 a^5 b p^2 + 
% 2 a^6 b p^2 + 3 a^3 b^2 p^2 - 3 a^4 b^2 p^2 \\ 
% + 2 a^4 b^3 p^2 - 
% a^3 b^4 p^2 + a^4 b^3 q - 2 a^3 b^4 q + a^2 b^5 q + 2 a^2 b^2 p q \\ 
% - 6 a^3 b^2 p q + 6 a^4 b^2 p q - 2 a^5 b^2 p q - 6 a^2 b^3 p q + 
% 6 a^2 b^4 p q - 2 a^2 b^5 p q \\ 
% - 2 a b^3 q^2 + 3 a^2 b^3 q^2 - 
% a^4 b^3 q^2 + b^4 q^2 + 6 a b^4 q^2 - 3 a^2 b^4 q^2 + 
% 2 a^3 b^4 q^2 \\ 
% - 3 b^5 q^2 - 6 a b^5 q^2 + 3 b^6 q^2 + 2 a b^6 q^2 - 
% b^7 q^2
%\end{multline} 
%and 
%\begin{multline}\label{eq:det1}
%	\frac{\pd p^2 D_1\,\pd q^2 D_1-(\pd p\pd q D_1)^2}{4 a^3 (a - b)^2 b^3} = \Det_1 \\
%	:=-2 + 9 a - 15 a^2 + 10 a^3 - 3 a^5 + a^6 + 9 b - 33 a b + 48 a^2 b - 
% 36 a^3 b + 15 a^4 b \\ 
% - 3 a^5 b - 15 b^2 + 48 a b^2 - 54 a^2 b^2 + 
% 24 a^3 b^2 - 3 a^4 b^2 + 10 b^3 - 36 a b^3 + 24 a^2 b^3 \\ 
% - 7 a^3 b^3 + 15 a b^4 - 3 a^2 b^4 - 3 b^5 - 3 a b^5 + b^6	
%\end{multline}
%for all $(a,b,p,q)\in\Om_1$. 
%
It suffices to show that $\Det_1<0$ for all $(a,b,p,q)\in\Om_1$, that is, for all $(a,b)\in\om_1$ -- since $\Det_1$ does not depend on $p$ or $q$. To that end, we shall first show that the maximum of $\Det_1$ over $\bar\om_1$ is not attained at any point $(a,b)\in\om_1$. Indeed, otherwise 
\begin{align*}
	\De_{1,0}\Det_1:=&\tfrac13\,\pd a\Det_1 \\ 
	=&3 - 10 a + 10 a^2 - 5 a^4 + 2 a^5 - 11 b + 32 a b - 36 a^2 b + 
 20 a^3 b \\ 
 &- 5 a^4 b + 16 b^2 - 36 a b^2 + 24 a^2 b^2 - 4 a^3 b^2 - 
 12 b^3 + 16 a b^3 \\ 
 &- 7 a^2 b^3 + 5 b^4 - 2 a b^4 - b^5 \\
\intertext{and}
 	\De_{0,1}\Det_1:=&\tfrac13\,\pd b\Det_1 \\ 
	=&3 - 11 a + 16 a^2 - 12 a^3 + 5 a^4 - a^5 - 10 b + 32 a b - 36 a^2 b \\
	 &+ 16 a^3 b - 2 a^4 b + 10 b^2 - 36 a b^2 + 24 a^2 b^2 - 7 a^3 b^2 + 
 20 a b^3 \\
 &- 4 a^2 b^3 - 5 b^4 - 5 a b^4 + 2 b^5 
\end{align*}
would both vanish at that point $(a,b)$, and then so would the resultant 
\begin{align*}
%R_1&(\De_{1,0}\Det_1,\De_{0,1}\Det_1)(b) \\
%=&243 (1 - b)^7 b^7(-3 + 21 b - 58 b^2 + 84 b^3 - 60 b^4 + 17 b^5) \\ 
%&\times(3168 - 13325 b + 
%   23645 b^2 - 22308 b^3 + 11916 b^4 - 3456 b^5 + 432 b^6)
%\intertext{and}  
R_b&(\De_{1,0}\Det_1,\De_{0,1}\Det_1)(a) \\ 
=&-243 (1 - a)^7 a^7 (-3 + 21 a - 58 a^2 + 84 a^3 - 60 a^4 + 17 a^5) \\ 
&\times(3168 - 13325 a + 
   23645 a^2 - 22308 a^3 + 11916 a^4 - 3456 a^5 + 432 a^6)  	
\end{align*}
of the polynomials $\De_{1,0}\Det_1$ and $\De_{0,1}\Det_1$ with respect to $b$. 
% and $b$; note that 
%$R_2&(\De_{1,0}\Det_1,\De_{0,1}\Det_1)(a)=-R_1&(\De_{1,0}\Det_1,\De_{0,1}\Det_1)(a)$. 
However, by Sturm's theorem (say), this resultant (which is a polynomial in one variable, $a$) does not have roots in the interval $(\frac12,1)$, whereas the condition $(a,b)\in\om_1$ implies $a\in(\frac12,1)$. 
This contradiction completes the verification that the maximum of $\Det_1$ over $\bar\om_1$ is not attained at any point of $\om_1$. 

Consider finally the values of $\Det_1$ on the boundary of $\om_1$. One has 
\begin{alignat*}{2}
	&\Det_1\big|_{b=a}&&=-2 + 18 a - 63 a^2 + 116 a^3 - 126 a^4 + 72 a^5 - 17 a^6<0\\
	&&&\qquad\qquad\qquad\qquad\qquad\qquad\qquad\qquad\qquad\qquad\text{ for }a\in(\tfrac12,1); \\ 
	&\Det_1\big|_{a=1}&&=b^3(-9 + 12 b - 6 b^2 + b^3)<0 
	\quad\text{ for }b\in(0,1); \\ 
	&\Det_1\big|_{a=1-b}&&=-9 (1-b)^3 b^3<0\quad\text{ for }b\in(0,1). 
\end{alignat*}
Thus, $\Det_1$ is no greater than $0$ on the boundary of $\om_1$ and does not attain its maximum over $\bar\om_1$ at any point $(a,b)\in\om_1$. It follows that $\Det_1<0$ for all $(a,b)\in\om_1$, that is, for all $(a,b,p,q)\in\Om_1$. 
Hence, the case $k=1$ of Lemma~\ref{lem:0<p,q<1} follows as well. 
\end{proof}

In view of Lemma~\ref{lem:0<p,q<1}, it remains to verify, for each $k\in\{0,1,2,3\}$, that $D_k\ge0$ for all $(a,b)\in\om_k$ and all $(p,q)\in[0,1]^2$ such that either $p\in\{0,1\}$ or $q\in\{0,1\}$. Thus, for each $k\in\{0,1,2,3\}$, one has to consider $4$ possibilities: $p=0$, $p=1$, $q=0$, and $q=1$, which results in $4\times4=16$ subcases. These 16 subcases will each be considered in one of the corresponding 16 lemmas below; some of these lemmas are very simple, and some are rather complicated. 
Recall notation \eqref{eq:D_;}. 
Details of calculations in the proofs of these 16 lemmas can be found in the corresponding Mathematica notebooks and their pdf copies. The names $\text{der0p0.nb (der0p0.pdf), \dots, der3q1.nb (der3q1.pdf)}$ of these notebooks (and pdf files) correspond to the polynomials $D_{0;p=0},\dots,D_{3;q=1}$, 
%(in $a,b,p$ or in $a,b,q$) 
whose nonnegativity is stated and proved in these lemmas. 
E.g., the ``root'' $\text{der0p0}$ of the names of the files $\text{der0p0.nb}$ and  $\text{der0p0.pdf}$ with extensions $\text{.nb}$ and  $\text{.pdf}$ is obtained from the symbol $D_{0;p=0}$ by replacing there $D$ by der and removing ``;'' and ``$=$''. 

%There, we shall use notation such as the following: 
%%\newcommand{\D}[2]{D_{#1;#2}}
%\begin{equation*}
%	D_{k;p=\de}:=D_k\big|_{p=\de},\quad D_{k;q=\vp}:=D_k\big|_{q=\vp},\quad D_{k;p=\de,q=\vp}:=D_k\big|_{p=\de,q=\vp};  
%\end{equation*}
%sometimes in such notation we shall use, instead of $D_k$, a modified version $\tD_k$ of $D_k$, which differs from $D_k$ by a factor which is manifestly positive in the corresponding context. 

\begin{lemma}\label{lem:0p0}
$D_{0;p=0}\ge0$ for all $(a,b)\in\om_0$ and $q\in(0,1)$. 
\end{lemma} 

\begin{proof}[Proof of Lemma~\ref{lem:0p0}] 
In view of \eqref{eq:D_0}, $D_{0;p=0}=b q [(a - b)^2 + 6 b q (1 - b)]$, which is manifestly nonnegative for all $(a,b)\in\om_0$ and $q\in[0,1]$.   
\end{proof}

\begin{lemma}\label{lem:0p1}
$D_{0;p=1}\ge0$ for all $(a,b)\in\om_0$ and $q\in(0,1)$. 
\end{lemma} 

\begin{proof}[Proof of Lemma~\ref{lem:0p1}] 
Assume indeed in this proof that $(a,b)\in\om_0$ and $q\in(0,1)$, unless otherwise indicated. 
In view of \eqref{eq:D_0}, 
\begin{equation*}
D_{0;p=1}=	b q (7 a^2+4 a b-12 a+b^2)-a (5 a^2+2 a b-6 a-b^2)+6 (1-b) b^2 q^2. 
\end{equation*}
So, $\pd q^2 D_{0;p=1}=12 (1- b) b^2>0$, whence $D_{0;p=1}$ is convex in $q$. 
Next, $D_{0;p=1,q=1}=(a - b)^2 [6 - 5 (a + b)]>0$ and $\big(\pd q{D_{0;p=1}}\big)\big|_{q=1}=(a - b) b [4 b + 7 (a + b) - 12]<0$, since $a+b<1$. 
Thus, Lemma~\ref{lem:0p1} follows. 
%
%So, it suffices to show that 
%\begin{equation}
%	D^{001}_{0;p=1,q=1}:=\tfrac1b\,(\pd q{D_{0;p=1}})\big|_{q=1}
%	=7 a^2 + 4 a b - 11 b^2-12 a + 12 b % \overset{\text{?}}\le0. 
%\end{equation}
%is $\le0$. Since $D^{001}_{0;p=1,q=1}$ is saddle-like in $a$ and $b$, the maximum of $D^{001}_{0;p=1,q=1}$ over all $(a,b)\in\bar\om_0$ is attained at some point $(a,b)$ on the boundary of $\om_0$, which is 
%\begin{equation}
%	\{(a,b)\colon0\le a\le1,b=0\}\cup\{(a,b)\colon0\le a=b\le\tfrac12\}\cup\{(a,b)\colon0\le b=1-a\le\tfrac12\}. 
%\end{equation}
%But $D^{001}_{0;p=1,q=1}\big|_{b=0}=- a (12 - 7 a)<0$, $D^{001}_{0;p=1,q=1}\big|_{b=a}=0$, and $D^{001}_{0;p=1,q=1}\big|_{b=1-a}=-(2 a -1) (1 + 4 a)\le0$ for $a\ge\frac12$.  
%So, the maximum of $D^{001}_{0;p=1,q=1}$ over all $(a,b)$ on the boundary of $\om_0$ and hence over all 
%$(a,b)\in\bar\om_0$ is $0$, which completes the proof of Lemma~\ref{lem:0p1}.    
\end{proof}

\begin{lemma}\label{lem:0q0}
$D_{0;q=0}\ge0$ for all $(a,b)\in\om_0$ and $p\in(0,1)$. 
\end{lemma} 

\begin{proof}[Proof of Lemma~\ref{lem:0q0}] 
In view of \eqref{eq:D_0}, $D_{0;q=0}=a p [(a - b)^2 + 6 a p (1 - a)]$, which is manifestly nonnegative for all $(a,b)\in\om_0$ and $p\in[0,1]$.   
\end{proof}

\begin{lemma}\label{lem:0q1}
$D_{0;q=1}\ge0$ for all $(a,b)\in\om_0$ and $p\in(0,1)$. 
\end{lemma} 

\begin{proof}[Proof of Lemma~\ref{lem:0q1}] % der0q1_alt.nb
Assume indeed in this proof that $(a,b)\in\om_0$ and $p\in(0,1)$, unless otherwise indicated. 
In view of \eqref{eq:D_0}, 
\begin{equation*}
D_{0;q=1}=	a^2 b + 6 b^2 - 2 a b^2 - 5 b^3 + a^3 p - 12 a b p + 4 a^2 b p + 
 7 a b^2 p + 6 a^2 p^2 - 6 a^3 p^2. 
\end{equation*}
If $D_{0;q=1}$ has a local extremum at 
some point $(a,b,p)\in\om_0\times(0,1)$, then at this point  
\begin{align*}
	D^{100}_{0;q=1}:=&\pd a{D_{0;q=1}}
	=2 a b - 2 b^2 + 3 a^2 p - 12 b p + 8 a b p + 7 b^2 p + 12 a p^2 \\ 
	&\qquad\qquad\qquad\qquad\qquad\qquad\qquad\qquad\qquad\qquad\qquad\quad - 18 a^2 p^2=0, \\ 
	D^{010}_{0;q=1}:=&\pd b{D_{0;q=1}}
	=a^2 + 12 b - 4 a b - 15 b^2 - 12 a p + 4 a^2 p + 14 a b p=0, \\ 
	D^{001}_{0;q=1}:=&\tfrac1a\,\pd p{D_{0;q=1}}
	=a^2 - 12 b + 4 a b + 7 b^2 + 12 a p - 12 a^2 p=0  
\end{align*}
(these three partial derivatives are denoted by \tt{D100der0q1}, \tt{D010der0q1}, \break 
\tt{D001der0q1}, respectively, in the files der0q1.nb and der0q1.pdf). 
So, (cf.\ \eqref{eq:R_n}),  
\begin{align*}
	\res^{101}:=&\frac{R_p(D^{100}_{0;q=1},D^{001}_{0;q=1})(a,b)}{6 a^2 (a - b)} \\ 
	=&-4 a^2 + 3 a^3 + 96 b - 148 a b + 67 a^2 b - 112 b^2 + 49 a b^2 + 
 49 b^3=0, \\ 
	\res^{011}:=&\frac{R_p(D^{010}_{0;q=1},D^{001}_{0;q=1})(a,b)}{2 a (a - b)}
	 =-12 a + 8 a^2 + 36 b - a b - 49 b^2=0, \\ 
	\res:=&R_a(\res^{101},\res^{011})(b)=864 b (7 b-4) (7 b-1)^2 (315 b^2-168 b-16)=0.  
\end{align*}
The latter equation, $\res=0$, together with the condition $(a,b)\in\om_0$ (which implies $0<b<\frac12$), yields $b=\frac17$. 
However, $\res^{101}\big|_{b=1/7}=\frac37 (1 - a)^2 (27 + 7 a)\ne0$. 
So, $D_{0;q=1}$ does not have a local extremum at 
any point $(a,b,p)\in\om_0\times(0,1)$. 

It remains to show that $D_{0;q=1}\ge0$ for all $(a,b,p)\in\bar\om_0\times[0,1]$ such that 
$p\in\{0,1\}$ or $(a,b)$ is on the boundary of $\om_0$; the latter, boundary condition on $(a,b)$ means that 
%
%for $p\in\{0,1\}$ and $(a,b)$ on the boundary of $\om_0$, that is, for all $(a,b)$ satisfying the condition 
\begin{equation*}
	(0\le a\le1,b=0)\ \orl\ (0\le a=b\le\tfrac12)\ \orl\ (0\le b=1-a\le\tfrac12). 
\end{equation*}

Now indeed, $D_{0;q=1}\big|_{p=0}=b [(a - b)^2 + 6 b (1 - b)]\ge0$, $D_{0;q=1}\big|_{p=1}=(a - b)^2 [6 - 5 (a + b)]\ge0$, 
$D_{0;q=1}\big|_{b=0}=a^2 p [a + 6 p (1 - a)]\ge0$, and $D_{0;q=1}\big|_{b=a}=6 (1 - a) a^2 (1 - p)^2\ge0$.  
It remains to show that 
\begin{align*}
	D_{0;q=1,b=1-a}:=&D_{0;q=1}\big|_{b=1-a} \\
	=&1 + a - 4 a^2 + 2 a^3 - a (5 - 2 a - 4 a^2) p + 6 (1- a) a^2 p^2\overset{\text{?}}\ge0; 
\end{align*}
this expression is denoted by \tt{der0q1b1a} in in the files der0q1.nb and der0q1.pdf. 
If $D_{0;q=1,b=1-a}$ has a local extremum at 
some point $(a,p)\in(0,1)^2$, then at this point  
\begin{align*}
	D^{10}_{0;q=1,b=1-a}:=&\pd a{D_{0;q=1,b=1-a}} \\ 
	=&1 - 8 a + 6 a^2 + (-5 + 4 a + 12 a^2) p + 6 (2 - 3 a) a p^2=0, \\  
	D^{01}_{0;q=1,b=1-a}:=&\tfrac1a\,\pd p{D_{0;q=1,b=1-a}}
	=-5 + 2 a + 4 a^2 + 12 (1 - a) a p=0    
\end{align*}
(these two partial derivatives are denoted by \tt{D10der0q1b1a}, \tt{D01der0q1b1a}, respectively, in the files der0q1.nb and der0q1.pdf). 
Hence,  
\begin{align*}
	%\res_a:=
	&\frac{R_a(D^{10}_{0;q=1,b=1-a},D^{01}_{0;q=1,b=1-a})(a,b)}{12 (1 - p)} 
%	\\ 
%	=&
	=-43 + 247 p - 229 p^2 + 57 p^3=0, \\ 
	%\res_p:=
	&\frac{R_p(D^{10}_{0;q=1,b=1-a},D^{01}_{0;q=1,b=1-a})(a,b)}{6 a^2} 
%	\\ 
%	 =&	
 =(-19 + 130 a - 208 a^2 + 96 a^3) \\
 &\qquad\qquad\qquad\qquad\qquad\qquad\qquad\qquad\qquad\qquad\qquad\qquad\qquad\times(-1 + 2 a) =0,   
\end{align*}
whence $p\in[\frac{21}{100},\frac{22}{100}]$ and $a\in\{\frac12\}\cup[\frac{20}{100}, \frac{21}{100}]
\cup[\frac{861}{1000}, \frac{862}{1000}]$, with the corresponding values of $D_{0;q=1,b=1-a}$ in 
$$[\tfrac{8797}{20000},\tfrac{19381}{40000}]\cup[\tfrac{1050018257}{1250000000},\tfrac{11209569}{12500000}]
\cup[\tfrac{23386805609}{312500000000},\tfrac{1055742276137}{5000000000000}]\subset[0,\infty);$$ 
here we used the standard method of interval calculus -- see e.g.\ \cite[Section~14.9.4]{yap00}. 

It remains to show that $D_{0;q=1,b=1-a}\ge0$ when $a$ or $p$ is in the set $\{0,1\}$. But indeed $D_{0;q=1,b=1-a}\big|_{a=0}=1\ge0$, $D_{0;q=1,b=1-a}\big|_{a=1}=p\ge0$, $D_{0;q=1,b=1-a}\big|_{p=0}=(1-a)[1 + 2 a (1-a)]\ge0$, and $D_{0;q=1,b=1-a}\big|_{p=1}=(1 - 2 a)^2\ge0$.  
Thus, the proof of Lemma~\ref{lem:0q1} is complete.  
\end{proof}

\begin{lemma}\label{lem:1p0}
$D_{1;p=0}\ge0$ for all $(a,b)\in\om_1$ and $q\in(0,1)$. 
\end{lemma} 

\begin{proof}[Proof of Lemma~\ref{lem:1p0}]
Assume indeed in this proof that $(a,b)\in\om_1$ and $q\in(0,1)$, unless otherwise indicated. 
Also just in this proof, let 
\begin{equation*}
	\tD_1:=\frac{D_1}{b^3q}. 
\end{equation*}
In view of \eqref{eq:D_1}, 
\begin{multline*}
\tD_{1;p=0}:=\frac{D_{1;p=0}}{b^3q}=a^2 (a-b)^2\\
+q (-a^4+2 a^3 b-3 a^2 b+3 a^2+2 a b^3-6 a b^2+6 a b-2 a-b^4+3 b^3-3 b^2+b), 	
\end{multline*}
which is of degree $1$ in $q$. So, without loss of generality $q\in\{0,1\}$. 
Next, $\tD_{1;p=0,q=0}=a^2 (a - b)^2>0$. 
Further, 
\begin{equation*}
\tD_{1;p=0,q=1}= (1 - b)^3 b - 2 a (1 - b)^3+ a^2 (3 - 3 b + b^2) 	
\end{equation*}
is convex in $a$, 
\begin{gather*}
\begin{aligned}
	\tD_{1;p=0,q=1}\big|_{a=1-b}=&(1 - b)^2[1 + 2 b (1 - b)]\ge0, \\ 
	\big(\pd a \tD_{1;p=0,q=1}\big)\big|_{a=1-b}=&2 (2 - b) (2 - b)\ge0, 
\end{aligned}	\\ 
	\tfrac1b\,\tD_{1;p=0,q=1}\big|_{a=b}=\tfrac12\,\big(\pd a \tD_{1;p=0,q=1}\big)\big|_{a=b}=-1 + 6 b - 6 b^2 + 2 b^3\ge0
\end{gather*} 
if $b\ge\frac12$ (which latter follows from the conditions $a=b$ and $(a,b)\in\bar\om_1$). 
Also, $(a,b)\in\bar\om_1$ implies that $a\ge (1-b)\vee b$.  
Thus, $\tD_{1;p=0,q=1}\ge0$, which completes the proof of Lemma~\ref{lem:1p0}. 
\end{proof}

\begin{lemma}\label{lem:1p1}
$D_{1;p=1}\ge0$ for all $(a,b)\in\om_1$ and $q\in[0,1]$. 
\end{lemma} 

\begin{proof}[Proof of Lemma~\ref{lem:1p1}] % der1p1.nb
Assume indeed in this proof that $(a,b)\in\om_1$ and $q\in(0,1)$. 
Recall \eqref{eq:D1;p=1} and \eqref{eq:D002D1;p=1};
in Mathematica notebook der1p1, we keep the notations \tt{D002der1p1} and \tt{D001der1p1q1} for $D^{002}_{1;p=1}$ and $D^{001}_{1;p=1,q=1}$, respectively.  
%In view of \eqref{eq:D_1}, 
%\begin{multline*}
%D_{1;p=1}=a^3 (a - 3 a^2 + 3 a^3 - a^4 - 2 b + 6 a b - 6 a^2 b + 2 a^3 b + 3 b^2 - 
% 3 a b^2 + a^2 b^2) \\
% + 
% a^2 b^2 (2 - 6 a + 6 a^2 - 2 a^3 - 6 b + a^2 b + 6 b^2 - 2 a b^2 - b^3) q \\
% - 
% b^3 (2 a - 3 a^2 + a^4 - b - 6 a b + 3 a^2 b - 2 a^3 b + 3 b^2 + 
%    6 a b^2 - 3 b^3 - 2 a b^3 + b^4) q^2.  	
%\end{multline*}
%Consider 
%\begin{equation*}
%	D^{002}_{1;p=1} := \frac{\pd q^2 D_{1;p=1} }{2 b^3}
%	=3 a^2 (1 - b) - (2 a - b) (1 - b)^3 + a^3 (2  b - a). 
%\end{equation*}
Note that $\pd b^2 D^{002}_{1;p=1}=-6 [1 + 2 (a - b)] (1 - b)<0$, so that $D^{002}_{1;p=1}$ is concave in $b$. 
Next, $D^{002}_{1;p=1}\big|_{b=1-a}=6 (1 - a) a^3>0$ and $\frac1a\,\big[D^{002}_{1;p=1}\big|_{b=a}\big]\big|_{a=u+1/2}
=\frac34 +%\break u 
(\frac32 - 3 u + 2 u^2)>\frac34>0$, because the condition $(a,b)\in\om_1$ implies $\frac12<a<1$ and hence $u:=a-\frac12>0$. 

So, $D^{002}_{1;p=1}>0$, whence $D_{1;p=1}$ is convex in $q$. 
Recall \eqref{eq:D1;p=1,q=1} and \eqref{eq:D001D1;p=1}. 
%Next, %$D_{1;p=1,q=1}=(a - b)^2[a^2 (1 - a)^3 + (1 - b)^3 b^2]>0$ and 
%\begin{gather*}
%D_{1;p=1}\big|_{q=1}=(a - b)^2[a^2 (1 - a)^3 + (1 - b)^3 b^2]>0, \\ 
%		D^{001}_{1;p=1,q=1} := \frac{\pd q D_{1;p=1}\big|_{q=1} }{b^2} \qquad\qquad\qquad\qquad\qquad\qquad\qquad\qquad\qquad\qquad \\
%	=a^4 (6 - b - 2 a) - 2 (1 - b)^3 b (2 a - b) - 2 a^3 (3 - b^2) + 
% a^2 (2 - b^3). 
%\end{gather*}
It remains to show that $D^{001}_{1;p=1,q=1}\le0$. 

Consider 
\begin{equation*}
D^{201}_{1;p=1,q=1} := \tfrac12\,\pd a^2 D^{001}_{1;p=1,q=1} =  2 - 20 a^3 + 6 a^2 (6 - b) - b^3 - 6 a (3 - b^2);  
\end{equation*}
this expression is denoted by \tt{D201der1p1q1} in Mathematica notebook der1p1. 
Note that $\pd b^2 D^{201}_{1;p=1,q=1} = 12 a - 6 b>0$, so that $D^{201}_{1;p=1,q=1}$ is convex in $b$. 
Next, $D^{201}_{1;p=1,q=1}\big|_{b=1-a}=-(1 - a)^2 (7a - 1)\le0$ for $a>\frac12$ and 
$D^{201}_{1;p=1,q=1}\big|_{b=a}=2 - 18 a + 36 a^2 - 21 a^3$ attains its maximum value $-0.102\ldots<0$ over all $a>\frac12$ at $a=\frac17 (4 + \sqrt2)=0.77\ldots$. 

So, $D^{201}_{1;p=1,q=1}<0$, which shows that $D^{001}_{1;p=1,q=1}$ is concave in $a$. 
Further, $D^{001}_{1;p=1,q=1}\big|_{a=b}=0$ and $\pd a D^{001}_{1;p=1,q=1}\big|_{a=b}=-6 (1 - b)^2 b^2<0$. 
It follows that indeed $D^{001}_{1;p=1,q=1}\le0$, which is what remained there to  
complete the proof of Lemma~\ref{lem:1p1}. 
\end{proof}

\begin{lemma}\label{lem:1q0}
$D_{1;q=0}\ge0$ for all $(a,b)\in\om_1$ and $p\in(0,1)$. 
\end{lemma} 

\begin{proof}[Proof of Lemma~\ref{lem:1q0}]
Assume indeed in this proof that $(a,b)\in\om_1$ and $p\in(0,1)$, unless otherwise indicated. 
Also just in this proof, let 
\begin{equation*}
	\tD_1:=\frac{D_1}{a^3p}. 
\end{equation*}
In view of \eqref{eq:D_1}, 
\begin{multline*}
\tD_{1;q=0}=(a - b)^2 b^2 + [a (1 - a)^3 + 
    b  (2 a^3 - 6 a^2 + 2 a b^2 - 3 a b + 6 a - b^3 + 3 b - 2)] p, 	
\end{multline*}
which is of degree $1$ in $p$. So, without loss of generality $p\in\{0,1\}$. 
Next, $\tD_{1;q=0,p=0}=(a - b)^2 b^2>0$. 
Further, 
\begin{equation*}
\tD_{1;q=0,p=1}= (1 - a)^3 ( a - 2 b) + (3 - 3 a + a^2) b^2 	
\end{equation*}
is convex in $b$, 
\begin{gather*}
\begin{aligned}
	\tD_{1;q=0,p=1}\big|_{b=1-a}=&(1 - a)^2[1 + 2 a (1-a)]\ge0, \\ 
	\big(\pd a \tD_{1;q=0,p=1}\big)\big|_{b=1-a}=&2 (2 - a) (1 - a)\ge0, 
\end{aligned}	
\end{gather*} 
and so, $\tD_{1;q=0,p=1}\ge0$, because the condition $(a,b)\in\om_1$ implies $b>1-a$.  
Therefore, $\tD_{1;q=0}\ge0$, which is equivalent to $D_{1;q=0}\ge0$ and thus 
completes the proof of Lemma~\ref{lem:1q0}. 
\end{proof}

\begin{lemma}\label{lem:1q1}
$D_{1;q=1}\ge0$ for all $(a,b)\in\om_1$ and $p\in(0,1)$. 
\end{lemma} 

\begin{proof}[Proof of Lemma~\ref{lem:1q1}]
Assume indeed in this proof that $(a,b)\in\om_1$ and $p\in(0,1)$, unless otherwise indicated. 
In view of \eqref{eq:D_1}, 
\begin{multline*}
D_{1;q=1}=b^3 (a^2 b^2 - 3 a^2 b + 3 a^2 + 2 a b^3 - 6 a b^2 + 6 a b - 2 a - 
    b^4 + 3 b^3 - 3 b^2 + b)  \\ 
- a^2 b^2 p (a^3 + 2 a^2 b - 6 a^2 - a b^2 + 6 a + 2 b^3 - 6 b^2 + 
    6 b - 2)  \\
- a^3 p^2 (a^4 - 2 a^3 b - 3 a^3 + 6 a^2 b + 3 a^2 - 2 a b^3 + 
    3 a b^2 - 6 a b - a + b^4 - 3 b^2 + 2 b).  	
\end{multline*}
If $D_{1;q=1}$ has a local extremum at 
some point $(a,b,p)\in\om_1\times(0,1)$, then at this point    
\begin{align*}
\pd a	D_{1;q=1}=&2 b^3 (a b^2 - 3 a b + 3 a + b^3 - 3 b^2 + 3 b - 1) \\ 
 &-a b^2 p (5 a^3 + 8 a^2 b - 24 a^2 - 3 a b^2 + 18 a + 4 b^3 - 
    12 b^2 + 12 b - 4) \\
    & - 
 a^2 p^2 (7 a^4 - 12 a^3 b - 18 a^3 + 30 a^2 b + 15 a^2 - 8 a b^3 + 
    12 a b^2  \\ 
    &\qquad\qquad\qquad\qquad\qquad\qquad\quad - 24 a b- 4 a + 3 b^4 - 9 b^2 + 6 b)=0, \\
\pd b	D_{1;q=1}=&b^2 (5 a^2 b^2-12 a^2 b+9 a^2+12 a b^3-30 a b^2+24 a b-6 a-7 b^4  \\
&\qquad\qquad\qquad\qquad\qquad\qquad\qquad\qquad\quad\   +18 b^3-15 b^2+4 b) \\ 
&+2 a^3 p^2 (a^3-3 a^2+3 a b^2-3 a b+3 a-2 b^3+3 b-1) \\ 
&-2 a^2 b p (a^3+3 a^2 b-6 a^2-2 a b^2+6 a+5 b^3-12 b^2+9 b-2)=0, \\
\frac{\pd p	D_{1;q=1}}{a^2}=&b^2 (2 - 6 a + 6 a^2 - a^3 - 6 b - 2 a^2 b + 6 b^2 + a b^2 - 2 b^3) \\
&- 
 2 a p (a^4 - 2 a^3 b - 3 a^3 + 6 a^2 b + 3 a^2 - 2 a b^3 + 3 a b^2 - 
    6 a b - a \\ 
    &\qquad\qquad\qquad\qquad\qquad\qquad\qquad\qquad\qquad\quad+ b^4 - 3 b^2 + 2 b)=0;   
\end{align*}
these three displayed expressions are denoted respectively by \tt{D100der1q1}, \tt{D010der1q1}, \tt{D001der1q1} in the Mathematica notebook der1q1.nb.   
So (cf.\ \eqref{eq:R_n}),  
\begin{align*}
	\res^{101}:=&\frac{R_p(\pd a	D_{1;q=1},\frac1{a^2}\,\pd p	D_{1;q=1})(a,b)}{a^2 b^3 (a - b) } \\ 
	=&3 a^9 b + a^8 (11 b^2 - 60 b + 24) - 
 a^7 (152 - 267 b - 6 b^2 + 35 b^3) \\
 &- 
 a^6 (55 b^4 - 348 b^3 + 369 b^2 + 486 b - 408) \\
 &- 
 a^5 (97 b^5 - 660 b^4 + 1764 b^3 - 1672 b^2 - 276 b + 600) \\
 &- 
 a^4 (33 b^6 - 492 b^5 + 2196 b^4 - 4304 b^3 + 3516 b^2 - 360 b - 
    520) \\
    &+ a^3 (5 b^7 + 126 b^6 - 1143 b^5 + 3724 b^4 - 5772 b^3 + 
    4068 b^2 - 732 b - 264) \\ 
    &- 
 3 a^2 (b - 1)^2 (5 b^6 - 10 b^5 + 80 b^4 - 316 b^3 + 468 b^2 - 
    224 b - 24) \\
    &+ 
 4 a (b - 1)^4 (b + 2) (5 b^4 - 14 b^3 + 30 b^2 - 26 b - 1) \\
 & - 
 12 (b - 1)^7 b (b + 1) (b + 2)=0, \\ 
	\res^{011}:=&\frac{R_p(\pd b	D_{1;q=1},\frac1{a^2}\,\pd p	D_{1;q=1})(a,b)}{2 a^2 b^2 (a - b) } \\ 
	 &=2 a^{10} b + a^9 (19 b^2 - 54 b + 18) - 
 a^8 (7 b^3 + 123 b^2 - 312 b + 120) \\
 &- 
 a^7 (88 b^4 - 282 b^3 - 135 b^2 + 794 b - 342) \\
 &+ 
 a^6 (52 b^5 + 366 b^4 - 1338 b^3 + 585 b^2 + 1014 b - 540) \\
 &- 
 a^5 (85 b^6 - 30 b^5 + 894 b^4 - 2947 b^3 + 2052 b^2 + 540 b - 
    510) \\
    &+ 3 a^4 (75 b^7 - 209 b^6 + 318 b^5 + 139 b^4 - 1108 b^3 + 
    942 b^2 - 54 b - 96) \\ 
    &- 
 a^3 (b - 1) (52 b^7 + 586 b^6 - 1817 b^5 + 2122 b^4 - 98 b^3 - 
    1484 b^2 \\ 
    &\qquad\qquad\qquad\qquad\qquad\qquad\qquad\qquad\qquad\qquad\qquad\qquad + 474 b + 90) \\ 
    &+ 
 6 a^2 (b - 1)^2 (17 b^7 - 34 b^6 + 164 b^5 - 300 b^4 + 202 b^3 + 
    26 b^2 - 37 b - 2) \\ 
    &- 
 6 a (b - 1)^4 b (13 b^5 - 17 b^3 + 72 b^2 - 26 b - 6) \\ 
 &+ 
 2 (b - 1)^6 b^2 (b + 2)^2 (7 b - 4)=0, 
 \end{align*}
 \begin{align*} 
	\res:=&R_b(\res^{101},\res^{011})(a) \\
	&=221184 (a-1)^{21} a^{21} (2 a^3-6 a^2+6 a-1) (13 a^3-21 a^2+12 a-2) \\ 
	&\times (242 a^7-2130 a^6+8562 a^5-20049 a^4+29184 a^3-25956 a^2\\
	&\qquad\qquad\qquad\qquad\qquad\qquad\qquad\qquad\qquad\qquad\qquad +12844 a-2700)^4 \\ 
	&\times (10232022093578922 a^{32}+243983422645932620 a^{31} \\
	&-4509200751103371792
   a^{30}+6473804776182225090 a^{29} \\ 
   &+392173262502922141262
   a^{28}-4886702832420734534706 a^{27} \\ 
   &+33179277660685865127615
   a^{26}-156920889820420563967402 a^{25} \\ 
   &+563853803520130467807750
   a^{24}-1609531040121454475809908 a^{23} \\ 
   &+3749969793954874952607178
   a^{22}-7259852838023387480456886 a^{21} \\ 
   &+11824823903860313224996785
   a^{20}-16347303258436435525821596 a^{19} \\ 
   &+19300470236420386651602360
   a^{18}-19542251444785239731292096 a^{17} \\ \break 
   &+17012372613520345095894435
   a^{16}-12747057132858406422499290 a^{15} \\ 
   &+8218546949212261087298286
   a^{14}-4552228300462560181850156 a^{13} \\ 
   &+2159681892325551379446354
   a^{12}-873593583957851406529086 a^{11} \\ 
   &+299376707230436244983721
   a^{10}-86180301652329538603686 a^9 \\ 
   &+20605976715300445636407
   a^8-4032357710602439708982 a^7 \\ 
   &+633335473568252968362
   a^6-77778352414976716710 a^5 \\ 
   &+7204715390026622547 a^4-478176202267116138
   a^3 \\ 
   &+21028578810363360 a^2-535896014633136 a+5818539184128)
	=0;  
\end{align*}
the three resultants above are denoted respectively by \tt{res101}, \tt{res011}, \tt{res} in the Mathematica notebook der1q1.nb.   
The latter equation, $\res=0$, has no roots $a\in(\frac12,1)$ -- which latter condition is implied by 
the condition $(a,b)\in\om_1$. 
So, $D_{1;q=1}$ does not have a local extremum at 
any point $(a,b,p)\in\om_1\times(0,1)$. 

So, it remains to show that $D_{1;q=1}\ge0$ for all $(a,b,p)\in\bar\om_1\times[0,1]$ such that 
$p\in\{0,1\}$ or $(a,b)$ is on the boundary of $\om_1$; the latter, boundary condition on $(a,b)$ means that 
\begin{equation*}
	(\tfrac12\le a=b\le1)\ \orl\ (0\le b\le1,a=1)\ \orl\ (\tfrac12\le a=1-b\le1). 
\end{equation*}

Now indeed, 
\begin{equation*}
D_{1;q=1,p=0}=D_{1;q=1}\big|_{p=0}=b^3 [(1 - b)^3 (b - 2 a) + a^2 (3 - 3 b + b^2)]	
\end{equation*}
is convex in $a$. At that, $D_{1;q=1,p=0}\big|_{a=1-b}=b^3(1 - 5 b^2 + 6 b^3 - 2 b^4)\ge0$ and 
$\pd a D_{1;q=1,p=0}\big|_{a=1-b}=2b^3(2 - b) (1 - b)\ge0$ for $b\in[0,\frac12]$. 
Also, for all $(a,b)\in\bar\om_1$ one has $a\ge1-b$ and \big($a=1-b\implies b\in[0,\frac12]$\big), whence $D_{1;q=1,p=0}\ge0$. 

Next, 
\begin{equation*}
	D_{1;q=1,p=1}=D_{1;q=1}\big|_{p=1}=(a - b)^2 [a^2 (1 - a)^3 + (1 - b)^3 b^2]\ge0 
\end{equation*}
for all $(a,b)\in\bar\om_1$. 

Further, for all $(a,b)\in\bar\om_1$ one has $a=b\implies b\in[\frac12,1]$, whence for all $p\in[0,1]$ 
\begin{equation*}
	D_{1;q=1,a=b}=D_{1;q=1}\big|_{a=b}=b^4 (1 - p)^2(2 b^3 - 6 b^2 + 6 b - 1)\ge0.  
\end{equation*}

It remains to consider the two more difficult boundary cases, with $a=1$ and with $a=1-b$.  

Consider now 
\begin{equation}\label{eq:tD1q1a1}
\begin{aligned}
	\tD_{1;q=1,a=1}:=&\tfrac1{b^2}\,D_{1;q=1}\big|_{a=1} \\ 
	&=5 b^4 - b^5 + p - 2 b^3 (4 + p) + b^2 (4 + 7 p - p^2) + 
 b (1 - 8 p + 2 p^2). 
\end{aligned} 
\end{equation}
If $\tD_{1;q=1,a=1}$ has a local extremum at 
some point $(b,p)\in(0,1)^2$, then at this point
\begin{align*}
\tD^{10}_{1;q=1,a=1}:=&\frac{\pd b\tD_{1;q=1,a=1}}{1-b} 
%	\\ 
%	=&
	=(b-1) (5 b^2-10 b-1)+2 (3 b-4) p+2 p^2=0, \\  
\tD^{01}_{1;q=1,a=1}:=&\pd p \tD_{1;q=1,a=1}
	=1 - 8 b + 7 b^2 - 2 b^3 + 2 (2 - b) b p=0;    
\end{align*}
these two displayed expressions are denoted respectively by \tt{D010der1q1a1}, \tt{D001der1q1a1} in the Mathematica notebook der1q1.nb.   
Hence, 
\begin{align*}
	%\res_p:=
	\frac{R_p(\tD^{10}_{1;q=1,a=1},\tD^{01}_{1;q=1,a=1})(b)}{2 (1 - b)^2 } 
	&=1 + 2 b - 59 b^2 + 106 b^3 - 58 b^4 + 10 b^5=0, \\  
		\frac{R_b(\tD^{10}_{1;q=1,a=1},\tD^{01}_{1;q=1,a=1})(p)}{4 (1 - p)^2 } 
	&=392 - 1304 p + 293 p^2 + 626 p^3 \\ 
	&\qquad\qquad\qquad\qquad\qquad\qquad+ 28 p^4 - 40 p^5=0;       
\end{align*}
these two resultants are denoted respectively by \tt{respder1q1a1}, \tt{resbder1q1a1} in the Mathematica notebook der1q1.nb. 
For the roots of these two resultants, one has  $b\in[\frac{18}{100},\frac{19}{100}]\cup[\frac{89}{100},\frac{90}{100}]$ and $p\in[\frac{34}{100},\frac{35}{100}]$, with the corresponding values of $\tD_{1;q=1,a=1}$ in the set 
$[\frac{1899988901}{10000000000},\frac{672192933}{2500000000}]
\cup[\frac{444839}{5000000},\frac{7710771351}{10000000000}]\subset[0,\infty)$. 
Therefore, $\tD_{1;q=1,a=1}\ge0$ at the interior critical points $(b,p)\in(0,1)^2$ of $\tD_{1;q=1,a=1}$. 
At that, $\tD_{1;q=1,a=1}\big|_{b=0}=p\ge0$, $\tD_{1;q=1,a=1}\big|_{b=1}=(1 - p)^2\ge0$, 
$\tD_{1;q=1,a=1}\big|_{p=0}=b+b^2(2 - b)^2 (1 - b) \ge0$, and $\tD_{1;q=1,a=1}\big|_{p=1}=(1 - b)^5\ge0$ for $(b,p)\in[0,1]^2$. 
Thus, $\tD_{1;q=1,a=1}\ge0$ and, in view of \eqref{eq:tD1q1a1}, 
\begin{equation}\label{eq:D_{1;q=1,a=1}>0}
\text{$D_{1;q=1,a=1}=D_{1;q=1}\big|_{a=1}\ge0$ for $(b,p)\in[0,1]^2$. }
\end{equation}

Finally, to complete the proof of Lemma~\ref{lem:1q1}, consider 
\begin{equation}\label{eq:tD1q1a1b}
\begin{aligned}
	\tD_{1;q=1,a=1-b}:=&\tfrac1{(1 - b)^2 b^2 }\,D_{1;q=1}\big|_{a=1-b} \\ 
	=&p + b (1 - 11 p + 6 p^2) + b^2 (2 + 14 p - 12 p^2) - 
 b^3 (2 + 4 p - 6 p^2),   
\end{aligned} 
\end{equation}
for $(b,p)\in[0,\frac12]\times[0,1]$; the latter displayed expression is denoted by \tt{der1q1a1b} in the Mathematica notebook der1q1.nb.   
If $\tD_{1;q=1,a=1-b}$ has a local extremum at 
some point $(b,p)\in(0,\frac12)\times(0,1)$, then at this point
\begin{align*}
\tD^{10}_{1;q=1,a=1-b}:=&\pd b\tD_{1;q=1,a=1-b} 
	=1 - 11 p + 6 p^2 + b (4 + 28 p - 24 p^2) \\ 
	&\qquad\qquad\qquad\qquad\qquad\qquad\qquad\qquad- 6 b^2 (1 + 2 p - 3 p^2)=0, \\  
\tD^{01}_{1;q=1,a=1-b}:=&\frac{\pd p \tD_{1;q=1,a=1-b}}{1-b}
	=1 - 2 b (5 - 6 p) + b^2 (4 - 12 p)=0;     
\end{align*}
these two displayed expressions are denoted respectively by \tt{D010der1q1a1b}, \tt{D001der1q1a1b} in the Mathematica notebook der1q1.nb.
Hence,  
\begin{align*}
	%\res_p:=
	\frac{R_p(\tD^{10}_{1;q=1,a=1-b},\tD^{01}_{1;q=1,a=1-b})(b)}{6 (1 - b)^2 (1 - 2 b) } 
	&=1 + 2 b - 80 b^2 + 96 b^3=0, \\  
		\frac{R_b(\tD^{10}_{1;q=1,a=1-b},\tD^{01}_{1;q=1,a=1-b})(p)}{12 (1 - p) } 
	&=-43 + 247 p - 229 p^2 + 57 p^3=0,      
\end{align*}
whence $b\in[\frac{13}{100},\frac{14}{100}]$ and $p\in[\frac{21}{100},\frac{22}{100}]$, with the corresponding values of $\tD_{1;q=1,a=1-b}$ in the set 
$[\frac{71711139}{625000000},\frac{213807761}{1250000000}]\subset[0,\infty)$; 
the latter two resultants are denoted respectively by \tt{respder1q1a1b}, \tt{resbder1q1a1b} in the Mathematica notebook der1q1.nb. 
Therefore, $\tD_{1;q=1,a=1-b}\ge0$ at the interior critical points $(b,p)\in(0,\frac12)\times(0,1)$ of $\tD_{1;q=1,a=1-b}$. 
At that, %\break 
$\tD_{1;q=1,a=1-b}\big|_{b=0}=p\ge0$, $\tD_{1;q=1,a=1-b}\big|_{b=\frac12}=\frac34(1 - p)^2\ge0$, 
$\tD_{1;q=1,a=1-b}\big|_{p=0}=b+2b^2(1 - b) \ge0$, and $\tD_{1;q=1,a=1-b}\big|_{p=1}=(1 - 2b)^2\ge0$ for $(b,p)\in[0,\frac12]\times[0,1]$. 
Thus, $\tD_{1;q=1,a=1-b}\ge0$ and, in view of \eqref{eq:tD1q1a1b}, 
$D_{1;q=1,a=1-b}=D_{1;q=1}\big|_{a=1-b}\ge0$ for $(b,p)\in[0,\frac12]\times[0,1]$. 
 
This  
completes the proof of Lemma~\ref{lem:1q1}. 
\end{proof}

\begin{lemma}\label{lem:2p0}
$D_{2;p=0}\ge0$ for all $(a,b)\in\om_2$ and $q\in(0,1)$. 
\end{lemma} 

\begin{proof}[Proof of Lemma~\ref{lem:2p0}]
Assume indeed in this proof that $(a,b)\in\om_2$ and $q\in(0,1)$, unless otherwise indicated. 
Also just in this proof, let 
\begin{equation*}
	\tD_2:=\frac{D_2}{bq}. 
\end{equation*}
In view of \eqref{eq:D_2}, 
\begin{multline*}
\tD_{2;p=0}:=\frac{D_{2;p=0}}{bq}=-2 a + 3 a^2 + b - 3 a^2 b + a^2 b^2 + (2 a - b) b (3 - 3 b + b^2) q, 	
\end{multline*}
which is of degree $1$ in $q$. So, without loss of generality $q\in\{0,1\}$. 

Next, let us show that $\tD_{2;p=0,q=0}=-2 a + b + a^2 (3 - 3 b + b^2)\ge0$. 
Since $\pd b \tD_{2;p=0,q=0}=1 + a^2 (2 b - 3)<1 + (2 b - 3)<0$, without loss of generality $b=1$. But $\tD_{2;p=0,q=0}\big|_{b=1}=(1 - a)^2\ge0$, which confirms that $\tD_{2;p=0,q=0}\ge0$. 

Finally, to complete the proof of Lemma~\ref{lem:2p0}, let us show that $\tD_{2;p=0,q=1}=(1 - b)^3 (b - 2 a) + a^2 (3 - 3 b + b^2)\ge0$. Note that $\tD_{2;p=0,q=1}$ is convex in $a$, and at that 
$\tD_{2;p=0,q=1}\big|_{a=1}=1 + (2 - b)^2 (1 - b) b\ge0$ 
and $\big(\pd a\tD_{2;p=0,q=1}\big)\big|_{a=1}=2 [2(1 - b^2) + b^3]\ge0$ for $b\in[0,1]$. 
Since the condition $(a,b)\in\om_2$ implies $a>1$, we conclude that indeed $\tD_{2;p=0,q=1}\ge0$, which completes the proof of Lemma~\ref{lem:2p0}. 
\end{proof}

\begin{lemma}\label{lem:2p1}
$D_{2;p=1}\ge0$ for all $(a,b)\in\om_2$ and $q\in(0,1)$. 
\end{lemma} 

\begin{proof}[Proof of Lemma~\ref{lem:2p1}]
Assume indeed in this proof that $(a,b)\in\om_2$ and $q\in(0,1)$. 
In view of \eqref{eq:D_2}, 
$%\begin{equation*}
D_{2;p=1}=f_1f_2, 	
$ %\end{equation*}
where $f_1:=f_1(q):=a^2 + b (b - 2 a) q$ and $f_2:=f_2(q):=1 - b (3 - 3 b + b^2) q$. 
Let us show that each of the two factors $f_1$ and $f_2$ is nonnegative. Since each of them 
is of degree $1$ in $q$, it it is enough to note that $f_1(0)=a^2\ge0$, $f_1(1)=(a-b)^2\ge0$, $f_2(0)=1\ge0$, and  $f_2(1)=(1-b)^3\ge0$. 
This completes the proof of Lemma~\ref{lem:2p1}. 
\end{proof}

\begin{lemma}\label{lem:2q0}
$D_{2;q=0}\ge0$ for all $(a,b)\in\om_2$ and $p\in(0,1)$. 
\end{lemma} 

\begin{proof}[Proof of Lemma~\ref{lem:2q0}]
Assume indeed in this proof that $(a,b)\in\om_2$ and $p\in(0,1)$. 
Also just in this proof, let 
\begin{equation*}
	\tD_2:=\frac{D_2}{a^2p}. 
\end{equation*}
In view of \eqref{eq:D_2}, 
\begin{multline*}
\tD_{2;q=0}:=\frac{D_{2;q=0}}{a^2p}=1 - 3 a + 3 a^2 - 2 a^2 b + a b^2 + a (3 - 3 a + 2 a b - b^2) p, 	
\end{multline*}
which is of degree $1$ in $p$. So, without loss of generality $p\in\{0,1\}$. 

Next, let us show that $\tD_{2;q=0,p=0}=1 + a^2 (3 - 2 b) - a (3 - b^2)\ge0$. 
Since $\pd b \tD_{2;q=0,p=0}=2 a (b - a)<0$, without loss of generality $b=1$. But $\tD_{2;q=0,p=0}\big|_{b=1}=(1 - a)^2\ge0$, which confirms that $\tD_{2;q=0,p=0}\ge0$. 

Finally, note that $\tD_{2;q=0,p=1}=1\ge0$. 

This completes the proof of Lemma~\ref{lem:2q0}. 
\end{proof}

\begin{lemma}\label{lem:2q1}
$D_{2;q=1}\ge0$ for all $(a,b)\in\om_2$ and $p\in(0,1)$. 
\end{lemma} 

\begin{proof}[Proof of Lemma~\ref{lem:2q1}]
Assume indeed in this proof that $(a,b)\in\om_2=(1,\infty)\times(0,1)$ and $p\in(0,1)$, unless otherwise indicated. 
In view of \eqref{eq:D_2}, 
\begin{multline*}
D_{2;q=1}=a^4 (2 b - 3) (p - 1) p - a^3 (b^2 - 3) (p - 1) p \\ 
+ 
 a^2 (3 b - 3 b^2 + b^3 + p - 6 b p + 6 b^2 p - 2 b^3 p) + 
 2 a (b - 1)^3 b - (b - 1)^3 b^2,   	
\end{multline*}
which is a polynomial in $a,b,p$ of degree $4$ in $a$. 

The correspondence 
\begin{equation}\label{eq:transf}
[1,\infty)\times[0,1]\times[0,1]=:\Th\ni(a,b,p)\longleftrightarrow %(A,b,p):=
(\tfrac1a,b,p)\in\hat\Th:=(0,1]\times[0,1]\times[0,1]	
\end{equation}
is bijective. It 
maps the unbounded set $[1,\infty)\times[0,1]\times[0,1]$ onto the %bounded 
set $(0,1]\times[0,1]\times[0,1]$, whose topological closure is of course the closed unit cube $[0,1]^3$, which is conveniently bounded. 
Moreover, 
\begin{equation*}
	\hat D_{2;q=1}:=A^4D_{2;q=1}\big|_{a=1/A}
\end{equation*}
is a polynomial in $A,b,p$ of degree $4$ in $A$, which is denoted by \tt{der2q1A} in the Mathematica notebook der2q1.nb.  
Clearly, $D_{2;q=1}\ge0$ for all $(a,b)\in\om_2$ and $p\in(0,1)$ iff $\hat D_{2;q=1}\ge0$ for all $(A,b,p)\in[0,1]^3$. 
However, the proof of the latter statement, that $\hat D_{2;q=1}$ for $(A,b,p)\in[0,1]^3$, is more convenient to conduct in the original variables $a,b,p$ and for the original polynomial $D_{2;q=1}$ -- except for the case when $A=0$ (informally corresponding to the ``case'' when $a\to\infty$), that is, except for the case when $(A,b,p)\in\{0\}\times[0,1]^2$. 

The latter, exceptional case, is in fact very easy. Indeed, $\hat D_{2;q=1}\big|_{A=0}=\break 
(3 - 2 b) (1- p) p\ge0$. 

Turning now away from the exceptional case, let $\inter E$ and $\bd E$ denote, respectively, the interior and the boundary in $\R^3$ of any set $E\subseteq\R^3$. Then, under the bijective correspondence \eqref{eq:transf}, $\inter\Th$ and $\bd\Th$ precisely correspond to $\inter\hat\Th$ and $(\bd\hat\Th)\setminus(\{0\}\times[0,1]^2)$, respectively. 
Thus, it remains to show that $D_{2;q=1}$ is nonnegative at all points 
$(a,b,p)\in\bd\Th$ and at all critical points $(a,b,p)\in\inter\Th$.   

If $D_{1;q=1}$ has a local extremum at 
some point $(a,b,p)\in\inter\Th=\om_1\times(0,1)$, then at this point   
\begin{align*}
\pd a	D_{2;q=1}=&a^2 p^2 (8 a b-12 a-3 b^2+9) \\ 
&-a p (8 a^2 b-12 a^2-3 a b^2+9 a+4 b^3-12 b^2+12 b-2) \\ 
&+2 b (a b^2-3 a b+3 a+b^3-3 b^2+3 b-1)=0, \\
\pd b	D_{2;q=1}=&2 a^3 p^2 (a-b)-2 a^2 p (a^2-a b+3 b^2-6 b+3) \\ 
&+(b-1)^2 (3 a^2+8 a b-2 a-5 b^2+2 b)=0, \\
\frac{\pd p	D_{2;q=1}}{a^2}=&1 - 3 a + 3 a^2 - 6 b - 2 a^2 b + 6 b^2 + a b^2 - 2 b^3 \\ 
&+ 
 2 a (3 - 3 a + 2 a b - b^2) p=0;  
\end{align*}
these three displayed expressions are denoted respectively by \tt{D100der2q1}, \tt{D010der2q1}, \tt{D001der2q1} in the Mathematica notebook der2q1.nb.
So (cf.\ \eqref{eq:R_n}),  
\begin{align*}
	\res^{101}:=&\frac{R_p(\pd a	D_{2;q=1},\frac1{a^2}\,\pd p	D_{2;q=1})(a,b)}{a^2} \\ 
	=&4 a^5 (3 - 2 b)^3 + 11 a^4 (2 b - 3)^2 (b^2 - 3) - 
 2 a^3 (2 b - 3) (5 b^4 - 30 b^2 - 4 b + 51) \\ 
 &+ 
 a^2 (35 b^6 - 192 b^5 + 429 b^4 - 552 b^3 + 417 b^2 - 24 b - 153) \\ 
 &- 
 4 a (b^2 - 3) (8 b^5 - 36 b^4 + 60 b^3 - 45 b^2 + 12 b + 3) \\ 
 &+ (b^2 - 
    3) (12 b^6 - 48 b^5 + 60 b^4 - 12 b^3 - 24 b^2 + 12 b + 1)=0, \\ 
	\res^{011}:=&\frac{R_p(\pd b	D_{2;q=1},\frac1{a^2}\,\pd p	D_{2;q=1})(a,b)}{2 a^2 } \\ 
	 =&- a^6 (2 b - 3)^2 + a^5 (2 b - 3) (4 b^2 - 3 b - 6) - 
 a^4 (b^2 - 3) (5 b^2 - 6 b - 3) \\ 
 &+ 
 a^3 (65 b^5 - 336 b^4 + 666 b^3 - 628 b^2 + 273 b - 36) \\ 
 &+ 
 a^2 (55 - 300 b + 312 b^2 + 400 b^3 - 930 b^4 + 588 b^5 - 124 b^6) \\ 
 &+ 
 a (64 b^7 - 228 b^6 + 48 b^5 + 698 b^4 - 984 b^3 + 420 b^2 - b - 
    18) \\ 
    &- 2 (b - 1)^2 b (5 b - 2) (b^2 - 3)^2=0, 
 \end{align*}
 \begin{align*} 
	\res:=&R_b(\res^{101},\res^{011})(a) \\
	&=-64 (a-1)^{15} a^2 (20 a^3-84 a^2+132 a-71)^4 \\ 
   &\times(177370200000 a^{17}-1259087910000 a^{16}+3655379856375
   a^{15} \\ 
   &-5140724114275 a^{14}+2772365717265 a^{13}-1047016242513
   a^{12} \\ 
   &+12034314857519 a^{11}-35530102580211 a^{10}+50669964404745
   a^9 \\ 
   &-45471398379553 a^8+30834049368285 a^7-17201730242769
   a^6 \\ 
   &+7858116810251 a^5-2832129730923 a^4+749567464029 a^3 \\ 
   &-130666203737
   a^2+12938625819 a-540087699)
	=0;   
\end{align*}
these three resultants are denoted respectively by \tt{res101p}, \tt{res011p}, \tt{res} in the Mathematica notebook der2q1.nb.
The only root $a\in(1,\infty)$ of the resultant $\res$ is the only root $a\in(1,\infty)$ of the equation $20 a^3-84 a^2+132 a-71=0$, which is $a=a_*:=\frac15\, (7 + 3\times 2^{-2/3} - 2^{5/3})$. 
Then the only root $b\in(0,1)$ of the equation $\res^{101}\big|_{a=a_*}=0$ is $b=b_*:=1 - 2^{-1/3}$. 
Next, for any $p\in(0,1)$ one has 
$D_{2;q=1}\big|_{a=a_*,b=b_*}=\frac1{200}\, (28 + 26\times2^{1/3} + 17\times2^{2/3})>0$. 
Thus, $D_{2;q=1}>0$ at all the possible critical points $(a,b,p)\in\inter\Th=\om_1\times(0,1)$.  

So, it remains to show that $D_{2;q=1}\ge0$ for all $(a,b,p)\in\bar\om_2\times[0,1]$ such that 
$p\in\{0,1\}$ or $(a,b)$ is on the boundary of $\om_2$; the latter, boundary condition on $(a,b)$ means that 
\begin{equation*}
	(a=1,0\le b\le1)\ \orl\ (b=0,a\ge1)\ \orl\ (b=1,a\ge1). 
\end{equation*}

Now indeed, 
\begin{equation*}
D_{2;q=1,p=0}=D_{2;q=1}\big|_{p=0}=b [(1 - b)^3 (b - 2 a) + a^2 (3 - 3 b + b^2)]	
\end{equation*}
is convex in $a$. At that, $D_{2;q=1,p=0}\big|_{a=1}=b(1 + 4 b - 8 b^2 + 5 b^3 - b^4)\ge0$ and 
$\pd a D_{2;q=1,p=0}\big|_{a=1}=2[2(1 - b^2) + b^3]\ge0$ for $b\in[0,1]$. 
So, $D_{2;q=1,p=0}\ge0$. 

Next, 
$%\begin{equation}
	D_{2;q=1,p=1}=D_{2;q=1}\big|_{p=1}=(a - b)^2 (1 - b)^3\ge0 
$. %\end{equation}
%for all $(a,b)\in\bar\om_1$. 

Further,  
\begin{equation*}
	D_{2;q=1,a=1}=D_{2;q=1}\big|_{a=1}=\tD_{1;q=1,a=1}\ge0,   
\end{equation*}
by \eqref{eq:tD1q1a1} and \eqref{eq:D_{1;q=1,a=1}>0}. 

Finally, 
$D_{2;q=1}\big|_{b=0}=a^2 p[1 + 3 a (a - 1) (1 - p)]\ge0$ and $D_{2;q=1}\big|_{b=1}=\break 
a^2 (1-p) [1 - p + p (a - 1)^2]\ge0$. 

This  
completes the proof of Lemma~\ref{lem:2q1}. 
\end{proof}

\begin{lemma}\label{lem:3p0}
$D_{3;p=0}\ge0$ for all $(a,b)\in\om_3$ and $q\in(0,1)$. 
\end{lemma} 

\begin{proof}[Proof of Lemma~\ref{lem:3p0}] 
In view of \eqref{eq:D_3}, $D_{3;p=0}=b^2 q [q a^2 + (1 - q) (a - b)^2]$, which is manifestly nonnegative for all $(a,b)\in\om_3$ and $q\in[0,1]$.   
\end{proof}

\begin{lemma}\label{lem:3p1}
$D_{3;p=1}\ge0$ for all $(a,b)\in\om_3$ and $q\in(0,1)$. 
\end{lemma} 

\begin{proof}[Proof of Lemma~\ref{lem:3p1}] 
In view of \eqref{eq:D_3}, $D_{3;p=1}=b^2 (1-q) [(1 - q) a^2 + q (a - b)^2]$, which is manifestly nonnegative for all $(a,b)\in\om_3$ and $q\in[0,1]$.   
\end{proof}

\begin{lemma}\label{lem:3q0}
$D_{3;q=0}\ge0$ for all $(a,b)\in\om_3$ and $p\in(0,1)$. 
\end{lemma} 

\begin{proof}[Proof of Lemma~\ref{lem:3q0}] 
In view of \eqref{eq:D_3}, $D_{3;q=0}=a^2 p [p b^2 + (1 - p) (a - b)^2]$, which is manifestly nonnegative for all $(a,b)\in\om_3$ and $q\in[0,1]$.   
\end{proof}

\begin{lemma}\label{lem:3q1}
$D_{3;q=1}\ge0$ for all $(a,b)\in\om_3$ and $p\in(0,1)$. 
\end{lemma} 

\begin{proof}[Proof of Lemma~\ref{lem:3q1}] 
In view of \eqref{eq:D_3}, $D_{3;q=1}=a^2 (1-p) [(1 - p) b^2 + p (a - b)^2]$, which is manifestly nonnegative for all $(a,b)\in\om_3$ and $p\in[0,1]$.   
\end{proof}

Now Lemma~\ref{lem:G'(0)>0} follows immediately from Lemmas~\ref{lem:0<p,q<1}--\ref{lem:3q1}. 

%\vspace{.2cm} \textsc{Acknowledgment}\ \ .

\bibliographystyle{acm}
\bibliography{C:/Users/ipinelis/Dropbox/mtu/bib_files/citations12.13.12}

\def\cprime{$'$} \def\polhk#1{\setbox0=\hbox{#1}{\ooalign{\hidewidth
  \lower1.5ex\hbox{`}\hidewidth\crcr\unhbox0}}}
  \def\polhk#1{\setbox0=\hbox{#1}{\ooalign{\hidewidth
  \lower1.5ex\hbox{`}\hidewidth\crcr\unhbox0}}}
  \def\polhk#1{\setbox0=\hbox{#1}{\ooalign{\hidewidth
  \lower1.5ex\hbox{`}\hidewidth\crcr\unhbox0}}} \def\cprime{$'$}
  \def\polhk#1{\setbox0=\hbox{#1}{\ooalign{\hidewidth
  \lower1.5ex\hbox{`}\hidewidth\crcr\unhbox0}}}
  \def\polhk#1{\setbox0=\hbox{#1}{\ooalign{\hidewidth
  \lower1.5ex\hbox{`}\hidewidth\crcr\unhbox0}}} \def\cprime{$'$}
  \def\cprime{$'$}
\begin{thebibliography}{10}

\bibitem{bennett}
{\sc Bennett, G.}
\newblock Probability inequalities for the sum of independent random variables.
\newblock {\em J. Amer. Statist. Assoc. 57}, 297 (1962), 33--45.

\bibitem{bent-ap}
{\sc Bentkus, V.}
\newblock On {H}oeffding's inequalities.
\newblock {\em Ann. Probab. 32}, 2 (2004), 1650--1673.

\bibitem{bent-64pp}
{\sc Bentkus, V., Kalosha, N., and van Zuijlen, M.}
\newblock On domination of tail probabilities of (super)martingales: explicit
  bounds.
\newblock {\em Liet. Mat. Rink. 46}, 1 (2006), 3--54.

\bibitem{billingsley}
{\sc Billingsley, P.}
\newblock {\em Convergence of probability measures}.
\newblock John Wiley \& Sons Inc., New York, 1968.

\bibitem{bouch-etal}
{\sc Boucheron, S., Bousquet, O., Lugosi, G., and Massart, P.}
\newblock Moment inequalities for functions of independent random variables.
\newblock {\em Ann. Probab. 33}, 2 (2005), 514--560.

\bibitem{cassier}
{\sc Cassier, G.}
\newblock Probl\`eme des moments sur un compact de {${\bf R}^{n}$} et
  d\'ecomposition de polyn\^omes \`a plusieurs variables.
\newblock {\em J. Funct. Anal. 58}, 3 (1984), 254--266.

\bibitem{collins98}
{\sc Collins, G.~E.}
\newblock Quantifier elimination for real closed fields by cylindrical
  algebraic decomposition.
\newblock In {\em Quantifier elimination and cylindrical algebraic
  decomposition ({L}inz, 1993)}, Texts Monogr. Symbol. Comput. Springer,
  Vienna, 1998, pp.~85--121.

\bibitem{cox-little-oshea}
{\sc Cox, D., Little, J., and O'Shea, D.}
\newblock {\em Ideals, varieties, and algorithms}, third~ed.
\newblock Springer, New York, 2007.
\newblock An introduction to computational algebraic geometry and commutative
  algebra.

\bibitem{dance}
{\sc Dance, C.~R.}
\newblock An inequality for the sum of independent bounded random variables.
\newblock {\em J. Theoret. Probab. 27}, 2 (2014), 358--369.

\bibitem{eaton1}
{\sc Eaton, M.~L.}
\newblock A note on symmetric {B}ernoulli random variables.
\newblock {\em Ann. Math. Statist. 41\/} (1970), 1223--1226.

\bibitem{eaton2}
{\sc Eaton, M.~L.}
\newblock A probability inequality for linear combinations of bounded random
  variables.
\newblock {\em Ann. Statist. 2\/} (1974), 609--613.

\bibitem{handelman85}
{\sc Handelman, D.}
\newblock Positive polynomials and product type actions of compact groups.
\newblock {\em Mem. Amer. Math. Soc. 54}, 320 (1985), xi+79.

\bibitem{handelman88}
{\sc Handelman, D.}
\newblock Representing polynomials by positive linear functions on compact
  convex polyhedra.
\newblock {\em Pacific J. Math. 132}, 1 (1988), 35--62.

\bibitem{hoeff63}
{\sc Hoeffding, W.}
\newblock Probability inequalities for sums of bounded random variables.
\newblock {\em J. Amer. Statist. Assoc. 58\/} (1963), 13--30.

\bibitem{ibr-shar-pos}
{\sc Ibragimov, R., and Sharakhmetov, S.}
\newblock The best constant in the {R}osenthal inequality for nonnegative
  random variables.
\newblock {\em Statist. Probab. Lett. 55}, 4 (2001), 367--376.

\bibitem{krivine64a}
{\sc Krivine, J.-L.}
\newblock Anneaux pr\'eordonn\'es.
\newblock {\em J. Analyse Math. 12\/} (1964), 307--326.

\bibitem{krivine64b}
{\sc Krivine, J.-L.}
\newblock Quelques propri\'et\'es des pr\'eordres dans les anneaux commutatifs
  unitaires.
\newblock {\em C. R. Acad. Sci. Paris 258\/} (1964), 3417--3418.

\bibitem{latala-moments}
{\sc Lata{\l}a, R.}
\newblock Estimation of moments of sums of independent real random variables.
\newblock {\em Ann. Probab. 25}, 3 (1997), 1502--1513.

\bibitem{loja}
{\sc {\L}ojasiewicz, S.}
\newblock Sur les ensembles semi-analytiques.
\newblock In {\em Actes du {C}ongr\`es {I}nternational des {M}ath\'ematiciens
  ({N}ice, 1970), {T}ome 2}. Gauthier-Villars, Paris, 1971, pp.~237--241.

\bibitem{marsh-ol}
{\sc Marshall, A.~W., and Olkin, I.}
\newblock {\em Inequalities: theory of majorization and its applications},
  vol.~143 of {\em Mathematics in Science and Engineering}.
\newblock Academic Press Inc. [Harcourt Brace Jovanovich Publishers], New York,
  1979.

\bibitem{maurer}
{\sc Maurer, A.}
\newblock A bound on the deviation probability for sums of non-negative random
  variables.
\newblock {\em JIPAM. J. Inequal. Pure Appl. Math. 4}, 1 (2003), Article 15, 6
  pp. (electronic).

\bibitem{okounkov-interview}
{\sc Mu{\~n}oz, V., and Persson, U.}
\newblock Interviews with three {F}ields medallists: {A}ndrei {O}kounkov.
\newblock {\em Newsletter of the {E}uropean {M}athematical {S}ociety}, 62
  (2006, December), 34--35.
\newblock \url{http://www.ams.org/notices/200703/comm-fields-interviews.pdf},
  pp.\ 1--2.

\bibitem{odlyzko-review}
{\sc Odlyzko, A.}
\newblock Review: {E}xperimental {M}athematics in {A}ction.
\newblock {\em Amer. Math. Monthly 118}, 10 (2011), 946--951.
\newblock \url{http://dx.doi.org/10.4169/amer.math.monthly.118.10.946}.

\bibitem{osek10}
{\sc Os{\c{e}}kowski, A.}
\newblock Sharp inequalities for sums of nonnegative random variables and for a
  martingale conditional square function.
\newblock {\em ALEA Lat. Am. J. Probab. Math. Stat. 7\/} (2010), 243--256.

\bibitem{T2}
{\sc Pinelis, I.}
\newblock Extremal probabilistic problems and {H}otelling's {$T^2$} test under
  a symmetry condition.
\newblock {\em Ann. Statist. 22}, 1 (1994), 357--368.

\bibitem{pin94}
{\sc Pinelis, I.}
\newblock Optimum bounds for the distributions of martingales in {B}anach
  spaces.
\newblock {\em Ann. Probab. 22}, 4 (1994), 1679--1706.

\bibitem{pin98}
{\sc Pinelis, I.}
\newblock Optimal tail comparison based on comparison of moments.
\newblock In {\em High dimensional probability ({O}berwolfach, 1996)}, vol.~43
  of {\em Progr. Probab.} Birkh\"auser, Basel, 1998, pp.~297--314.

\bibitem{pin99}
{\sc Pinelis, I.}
\newblock Fractional sums and integrals of {$r$}-concave tails and applications
  to comparison probability inequalities.
\newblock In {\em Advances in stochastic inequalities ({A}tlanta, {GA}, 1997)},
  vol.~234 of {\em Contemp. Math.} Amer. Math. Soc., Providence, RI, 1999,
  pp.~149--168.

\bibitem{asymm}
{\sc Pinelis, I.}
\newblock Exact inequalities for sums of asymmetric random variables, with
  applications.
\newblock {\em Probab. Theory Related Fields 139}, 3-4 (2007), 605--635.

\bibitem{pin-hoeff-arxiv-reftoAIHP}
{\sc Pinelis, I.}
\newblock On the {B}ennett-{H}oeffding inequality.
\newblock \url{http://arxiv.org/abs/0902.4058}; a shorter version appeared in
  \cite{pin-hoeff-published}, 2009.

\bibitem{positive}
{\sc Pinelis, I.}
\newblock Positive-part moments via the {F}ourier--{L}aplace transform.
\newblock {\em J. Theor. Probab. 24\/} (2011), 409--421.

\bibitem{p=3_publ}
{\sc Pinelis, I.}
\newblock Exact {R}osenthal-type inequalities for $p=3$, and related results.
\newblock {\em Statistics \& Probability Letters 83}, 12 (2013), 2634--2637.

\bibitem{pin-hoeff-published}
{\sc Pinelis, I.}
\newblock On the {B}ennett--{H}oeffding inequality.
\newblock {\em Annales de l'Institut Henri Poincar\'e, Probabilit\'es et
  Statistiques 50}, 1 (2014), 15--27.

\bibitem{bahr-esseen-AFA}
{\sc Pinelis, I.}
\newblock Best possible bounds of the von {B}ahr--{E}sseen type.
\newblock \url{http://arxiv.org/abs/1101.3286}, to appear in the Annals of
  Functional Analysis, 2015.

\bibitem{cones}
{\sc {Pinelis}, I.}
\newblock {Convex cones of generalized multiply monotone functions and the dual
  cones}.
\newblock {\em ArXiv e-prints\/} (Jan. 2015).
\newblock arXiv:1501.06599v1 [math.CA].

\bibitem{rosenthal_AOP}
{\sc Pinelis, I.}
\newblock Exact {R}osenthal-type bounds.
\newblock {\em Ann. Probab. 43}, 5 (2015), 2511--2544.

\bibitem{pin-sakh}
{\sc Pinelis, I.~F., and Sakhanenko, A.~I.}
\newblock Remarks on inequalities for probabilities of large deviations.
\newblock {\em Theory Probab. Appl. 30}, 1 (1985), 143--148.

\bibitem{pin-utev-exp}
{\sc Pinelis, I.~F., and Utev, S.~A.}
\newblock Sharp exponential estimates for sums of independent random variables.
\newblock {\em Theory Probab. Appl. 34}, 2 (1989), 340--346.

\bibitem{pourciau}
{\sc Pourciau, B.~H.}
\newblock Modern multiplier rules.
\newblock {\em Amer. Math. Monthly 87}, 6 (1980), 433--452.

\bibitem{voevodsky-SciAmer}
{\sc Rehmeyer, J.}
\newblock Voevodsky's {M}athematical {R}evolution.
\newblock {\em Scientific American}, October 1 (2013).
\newblock
  \url{http://blogs.scientificamerican.com/guest-blog/2013/10/01/voevodskys-mathematical-revolution/}.

\bibitem{rosenthal}
{\sc Rosenthal, H.~P.}
\newblock On the subspaces of {$L^{p}$} {$(p>2)$} spanned by sequences of
  independent random variables.
\newblock {\em Israel J. Math. 8\/} (1970), 273--303.

\bibitem{shenton}
{\sc Shenton, L.~R.}
\newblock Inequalities for the normal integral including a new continued
  fraction.
\newblock {\em Biometrika 41\/} (1954), 177--189.

\bibitem{tarski48}
{\sc Tarski, A.}
\newblock {\em A {D}ecision {M}ethod for {E}lementary {A}lgebra and
  {G}eometry}.
\newblock RAND Corporation, Santa Monica, Calif., 1948.

\bibitem{tyurinSPL}
{\sc Tyurin, I.~S.}
\newblock Some optimal bounds in the central limit theorem using zero biasing.
\newblock {\em Statist. Probab. Lett. 82}, 3 (2012), 514--518.

\bibitem{vanderWaerden}
{\sc van~der Waerden, B.~L.}
\newblock {\em Algebra. {V}ol 1}.
\newblock Translated by Fred Blum and John R. Schulenberger. Frederick Ungar
  Publishing Co., New York, 1970.

\bibitem{bahr65}
{\sc von Bahr, B., and Esseen, C.-G.}
\newblock Inequalities for the {$r$}th absolute moment of a sum of random
  variables, {$1\leq r\leq 2$}.
\newblock {\em Ann. Math. Statist 36\/} (1965), 299--303.

\bibitem{yap00}
{\sc Yap, C.~K.}
\newblock {\em Fundamental problems of algorithmic algebra}.
\newblock Oxford University Press, New York, 2000.

\end{thebibliography}
%\bibliography{C:/Users/Iosif/Dropbox/mtu/bib_files/citations12.13.12}
%\bibliography{C:/Users/Iosif/Documents/mtu_home01-30-10/bib_files/citations}

\end{document}